\documentclass[12pt]{amsart}

\usepackage{amssymb}
\usepackage{latexsym,array}
\usepackage{verbatim,euscript}
\usepackage{fullpage,color}

\input {cyracc.def}
\numberwithin{equation}{section}
\renewcommand{\theequation}{\thesection$\cdot$\arabic{equation}}

\usepackage[colorlinks=true,linkcolor=blue,urlcolor=my_color,citecolor=magenta]{hyperref}

\definecolor{my_color}{rgb}{0,0.5,0.5}
\definecolor{MIXT}{rgb}{0.4,0.3,0.6}

 \font\tencyr=wncyr10 %scaled \magstephalf
\font\tencyi=wncyi10 %scaled \magstephalf
\font\tencysc=wncysc10 %scaled \magstephalf
\def\rus{\tencyr\cyracc}
\def\rusi{\tencyi\cyracc}
\def\rusc{\tencysc\cyracc}

\newtheorem{thm}{Theorem}[section] 
\newtheorem{conj}[thm]{Conjecture}
\newtheorem{utv}[thm]{Claim}
\newtheorem*{utve}{Claim}
\newtheorem{lm}[thm]{Lemma}

\newtheorem{prop}[thm]{Proposition}%[chapter]

\theoremstyle{remark}
\newtheorem{rmk}[thm]{Remark}
\newtheorem*{rema}{Remark}

\theoremstyle{definition}
\newtheorem{ex}[thm]{Example}
\newtheorem*{ex-bn}{Example}
\newtheorem{df}{Definition}

\newenvironment{proof*}
{\noindent {\sl Proof.}\quad }{\hfill
$\square$}

\renewcommand{\theequation}{\thesection ${\cdot}$\arabic{equation}}

\newcommand {\ah}{{\mathfrak a}}

\newcommand {\ce}{{\mathfrak c}}

\newcommand {\fe}{{\mathfrak E}}

\newcommand {\g}{{\mathfrak g}}
\newcommand {\h}{{\mathfrak h}}
\newcommand {\ka}{{\mathfrak k}}
\newcommand {\el}{{\mathfrak l}}
\newcommand {\m}{{\mathfrak m}}

\newcommand {\p}{{\mathfrak p}}
\newcommand {\q}{{\mathfrak q}}

\newcommand {\es}{{\mathfrak s}}
\newcommand {\te}{{\mathfrak t}}

\newcommand {\z}{{\mathfrak z}}

\newcommand {\gln}{\mathfrak{gl}_n}
\newcommand {\sln}{\mathfrak{sl}_n}
\newcommand {\sltn}{\mathfrak{sl}_{2n}}

\newcommand {\spn}{\mathfrak{sp}_{2n}}

\newcommand {\sone}{\mathfrak{so}_{2n}}

\newcommand {\son}{\mathfrak{so}_{n}}
\newcommand {\fso}{\mathfrak{so}}

\newcommand {\eus}{\EuScript}

\newcommand {\esi}{\varepsilon}
\newcommand {\ap}{\alpha}

\newcommand {\lb}{\lambda}
\newcommand {\vp}{\varphi}

\newcommand {\cF}{{\mathcal F}}
\newcommand {\ck}{{\mathcal K}}
\newcommand {\MM}{{\mathcal M}}
\newcommand {\N}{{\mathcal N}}
\newcommand {\co}{{\mathcal O}}

\newcommand {\BZ}{{\mathbb Z}}

\newcommand {\BQ}{{\mathbb Q}}
\newcommand {\BR}{{\mathbb R}}

\newcommand {\CSS}{{\sf CSS}}
\newcommand {\CSA}{{\sf CSA}}

\newcommand {\ad}{{\mathrm{ad\,}}}

\newcommand {\codim}{{\mathrm{codim\,}}}

\newcommand {\Int}{{\mathrm{Int}}}
\newcommand {\Inv}{{\mathsf{Inv}}}
\newcommand {\Lie}{{\mathsf{Lie}}}
\newcommand {\Ker}{{\mathrm{Ker\,}}}
\newcommand {\Ima}{{\mathrm{Im\,}}}

\newcommand {\rk}{{\mathsf{rk}}}

\newcommand {\trdeg}{{\mathrm{trdeg\,}}}

\newcommand {\tri}{\mathfrak{sl}_2}
\newcommand {\GR}[2]{{\textrm{{\bf #1}}}_{#2}}
\newcommand {\GRt}[2]{{\tilde{\mathrm{{\bf #1}}}}_{#2}}
\newcommand {\ov}{\overline}
\newcommand {\un}{\underline}

\newcommand {\beq}{\begin{equation}}
\newcommand {\eeq}{\end{equation}}

\renewcommand{\le}{\leqslant}
\renewcommand{\ge}{\geqslant}

\newcommand{\vecs}{{\vec{\boldsymbol{\sigma}}}}

\newcommand {\bbk}{\Bbbk}%{\mathbb F}%
\begin{document}
\setlength{\parskip}{1pt plus 2pt minus 0pt}
\hfill {\scriptsize October 14. 2012}%} 
\vskip1.5ex

\title[Commuting involutions, commuting varieties,  and simple Jordan algebras]
{Commuting involutions of Lie algebras, commuting varieties,  and simple Jordan algebras}
\author[D.\,Panyushev]{Dmitri I. Panyushev}
\address[]{
Institute for Information Transmission Problems of the R.A.S., 
\hfil\break\indent B. Karetnyi per. 19, Moscow 
127994, Russia
}
\email{panyushev@iitp.ru}
%\urladdr{\url{http://www.mccme.ru/~panyush}}
\begin{abstract}
Let $\sigma_1$ and $\sigma_2$ be commuting involutions of a connected 
reductive algebraic group $G$ with $\g=\Lie(G)$. Let $\g=\bigoplus_{i,j=0,1}\g_{ij}$ be the corresponding $\BZ_2\times \BZ_2$-grading. 
If $\{\ap,\beta,\gamma\}=\{01,10,11\}$, then $[\ ,\ ]:\g_\ap\times\g_\beta\to \g_\gamma$, and the zero-fibre of this bracket is called a $\vecs$-commuting variety. The commuting variety of $\g$ and
commuting varieties related to one involution are particular cases of this construction. We develop 
a general theory of such varieties and point out some cases, when they have especially good 
properties. If $G/G^{\sigma_1}$ is a Hermitian symmetric space of tube type, then one can 
find three conjugate pairwise commuting involutions $\sigma_1,\sigma_2$, and 
$\sigma_3=\sigma_1\sigma_2$. In this case, any $\vecs$-commuting variety is isomorphic to the commuting variety of the simple Jordan algebra associated with $\sigma_1$. As an application, we show that if $\eus J$ is the Jordan algebra of symmetric matrices, then the product map
$\eus J\times \eus J\to\eus J$ is equidimensional; while for all other simple Jordan algebras equidimensionality fails.
\end{abstract}
\maketitle

\tableofcontents
\section*{Introduction}
\noindent
The ground field $\bbk$ is algebraically closed and $\mathsf{char\,}\bbk=0$. 
Let $G$ be a connected reductive algebraic group with $\Lie(G)=\g$. In 1979, Richardson proved that any pair of commuting elements of $\g$ can be approximated by pairs of commuting semisimple elements~\cite{ri79}. More precisely, if $\te\subset \g$ is a Cartan subalgebra (\CSA\ for short), then
\beq   \label{eq:comm-var-g}
     \{(x,y)\in\g\times\g \mid [x,y]=0\}=\ov{G{\cdot}(\te\times\te)},
\eeq
where `bar' means the Zariski closure. The LHS is called the {\it commuting variety\/} of $\g$,
denoted $\fe(\g)$. That is, $\fe(\g)$ is the zero-fibre of the multiplication map
$\g\times\g\stackrel{[\ ,\ ]}{\longrightarrow}\g$. It follows from \eqref{eq:comm-var-g} that $\fe(\g)$ is 
irreducible and $\dim\fe(\g)=\dim\g+\rk\,\g$. For arbitrary Lie algebras, e.g. for Borel subalgebras of $\g$, the commuting variety 
can be reducible~\cite[p.\,237]{vasc}.

There are several directions for generalising Richardson's work. 
\\[.7ex]
\un{First}, for given subvarieties  $U,V\subset \g$, one can consider the restriction of $[\ ,\ ]$ to 
$U\times V$ and study properties of $\fe(\g)\cap( U\times V)$.  For instance: 

{\bf --}  \ Let  $\sigma$ be an involution of $\g$ with the corresponding $\BZ_2$-grading $\g=\g_0\oplus\g_1$.
Taking $U=V=\g_1$ yields the commuting variety %related to $\sigma$,
$\fe(\g_1):=\fe(\g)\cap(\g_1\times\g_1)$, which was considered first in~\cite{compos94}.
Here the structure of $\fe(\g_1)$ heavily depends on $\sigma$. If $\g_1$ contains a \CSA\ of $\g$, then $\fe(\g_1)$ is an irreducible normal complete intersection \cite{compos94}. At the other extreme, if the symmetric space $G/G_0$ is of rank 1,
then $\fe(\g_1)$ is often reducible. In \cite{PY07}, the question of
irreducibility of $\fe(\g_1)$ is resolved for all but three involutions of simple Lie algebras, and the
remaining cases are settled in \cite{bulois}. It seems, however, that there is no simple rule to
distinguish the involutions for which $\fe(\g_1)$ is irreducible. 

{\bf --} \ Another natural possibility is to take $U=V=\N$, where $\N$ is the set of nilpotent elements of $\g$.
This leads to the {\it nilpotent commuting variety\/} of $\g$, $\fe(\N)$, which is often reducible. However,
$\fe(\N)$ is equidimensional, $\dim\fe(\N)=\dim\g$, and the structure of irreducible components is well understood~\cite{premet}.

{\bf --} \ An interesting situation with $U\ne V$ occurs if $\g=\oplus_{i\in\BZ}\g(i)$ is 
$\BZ$-graded, $U=\g(i)$, and $V=\g(-i)$, see \cite[Sect.\,3]{conorm99}. 
\\[.7ex]
\un{Second}, one may look at commuting varieties related to other types of algebras. If $\eus A$ is any algebra, then $\fe(\eus A)$ is defined to be the zero fibre of the multiplication map
$\eus A\times \eus A\to \eus A$.
%One of the first questions that comes in mind is (at least, in my mind): 
It is a natural task to study the commuting variety of a simple Jordan algebra. 
As far as I know, this problem has not been addressed before.

In this article, we elaborate on both directions outlined above. We study 
certain ``commuting varieties'' associated with $\BZ_2\times\BZ_2$-gradings
%a pair of commuting involutions $\sigma_1,\sigma_2$ 
of $\g$ (the first direction). It turns out that, for some gradings,
%choice of  $\sigma_i$'s, 
these new commuting varieties are isomorphic to the commuting variety of
simple Jordan algebras (the second direction). 
To describe our results more precisely, we need some notation.
Let $\sigma_1$ and $\sigma_2$ be different commuting involutions of a connected reductive 
algebraic group $G$. This yields a $\mathbb Z_2\times \mathbb Z_2$-grading of 
$\g$:
\beq   
\label{eq:decomp0}
\g=\bigoplus_{i,j=0,1}\g_{ij}, \ \text{ where }\ \g_{ij}=\{x\in\g\mid   \sigma_1(x)=(-1)^ix \ \ \& \ \ 
\sigma_2(x)=(-1)^jx\}.
\eeq
Then $\sigma_1,\sigma_2$, and $\sigma_3=\sigma_1\sigma_2$ are  pairwise 
commuting involutions, 
%this decomposition possesses an $\mathbb S_3$-symmetry, 
and following \cite{vergne} we say that \eqref{eq:decomp0}
is a {\it quaternionic decomposition} of $\g$. 
For, if $(\ap,\beta,\gamma)$ is any permutation of the set of indices $\{01,10,11\}$, then 
$[\g_{00},\g_\ap] \subset \g_\ap$ and  $[\g_\ap,\g_\beta] \subset \g_\gamma$.
The conjugacy classes of pairs of commuting 
involutions are classified, see \cite{kollross} and references therein. Therefore, it is not difficult
to write down explicitly all the quaternionic decompositions of simple Lie algebras.
This article is a continuation of 
\cite{comm-inv1}, where we developed some theory on Cartan subspaces related to 
\eqref{eq:decomp0} and studied invariants of degenerations
of isotropy representations involved.

Set $\vecs=(\sigma_1,\sigma_2, \sigma_3)$, and let $G_{00}$ denote the connected subgroup of
$G$ with Lie algebra $\g_{00}$.
A $\vecs$-{\it commuting variety\/} is the zero-fibre of the bracket
$[\ ,\ ]: \g_\ap\times\g_\beta\longrightarrow \g_\gamma$. Associated with 
\eqref{eq:decomp0}, one has three essentially different such varieties that are parameterised  
by the choice of $\gamma\in \{01,10,11\}$.  All these mappings are $G_{00}$-equivariant, and 
all $\vecs$-commuting varieties are $G_{00}$-varieties.
The above-mentioned  varieties $\fe(\g_1)$ can be obtained 
as a special case of this construction, see Example~\ref{ex:g=s+s}.
We usually stick to one particular choice of the commutator,
$\vp:\g_{10}\times\g_{11}\to \g_{01}$, and try to realise what assumptions on  
$\vecs$ imply good properties of $\fe:=\vp^{-1}(0)$ and other fibres of $\vp$.
Clearly,  $\vp$ can be regarded as a quadratic map from $\g_{1\star}:=\g_{10}\oplus\g_{11}$
to $\g_{01}$. Let $\ce_{1\star}$ be a Cartan subspace (=\CSS) in $\g_{1\star}$.
Say that $\ce_{1\star}$ is {\it homogeneous\/} if it is $\sigma_2$-stable (or, equivalently,
$\sigma_3$-stable), i.e., if $\ce_{1\star}=\ah_{10}\oplus\ah_{11}$ with 
$\ah_{1j}\subset \g_{1j}$. We prove that

{\textbullet}  \  if $\ce_{1\star}$ is a homogeneous \CSS, then the closure of
$G_{00}{\cdot}\ce_{1\star}$ is an irreducible component of $\fe$
(Theorem~\ref{thm:irr-comp-Fe}). (Such irreducible components are said to be {\it standard}). 
However, there can be several standard component, of different dimension; and there can also exist some ``non-standard'' irreducible components. 
% Section~\ref{sect2}).

{\textbullet}  \  All homogeneous \CSS\ in $\g_{1\star}$ are $G_{00}$-conjugate (i.e., $\fe$ has 
only one standard component) if and only if 
$\dim\ce_{1\star}= \dim\ce_{10}+\dim\ce_{11}$, where $\ce_{1j}$ are \CSS\ in $\g_{1j}$ 
(Theorem~\ref{thm:krit-odno-homog-CSS}).

{\textbullet}  \  $\vp$ is dominant if and only if there exist $x\in\g_{10}, y\in\g_{11}$ such that 
$\z_\g(x)_{01}\cap\z_\g(y)_{01}=\{0\}$. 
\\[.6ex]
However, one cannot expect really good properties for $\vp$ and $\fe$ without extra assumptions.
One natural assumption is that  %are not totally unrelated, i.e., that 
some of involutions in $\vecs$ are conjugate. Another possibility is that some of the $\sigma_i$'s 
possess prescribed properties.
Our more specific results are:

%{\small $1{\blacktriangleright}$}  
(1) \  If $\sigma_1,\sigma_2$ are conjugate, then $\vp$ is surjective
and $\dim\vp^{-1}(\xi)\ge \dim\g_{11}$ for all $\xi\in\g_{01}$ (Proposition~\ref{lm:dyad-onto}). 
We also provide a method for
detecting subvarieties of $\fe$ whose dimension is larger than $\dim\g_{11}$. This exploits certain restricted root systems related to decomposition~\eqref{eq:decomp0}, see Section~\ref{sect4}.

%{\small $2{\blacktriangleright}$}  
(2) \  If $\sigma_1,\sigma_2$ are involutions of maximal rank
(hence they are conjugate), then 
$\vp$ is surjective and equidimensional, each irreducible component of $\fe$ is standard, and
the scheme $\vp^{-1}(0)$ is a reduced complete intersection (Theorem~\ref{thm:main-EQ}).

%{\small $3{\blacktriangleright}$}  
(3) \  Let $\g$ be simple and $\sigma$ a Hermitian involution
(i.e., $\g^\sigma$ is not semisimple). If the Hermitian symmetric space $G/G^\sigma$ is of tube
type, then there exists a commuting triple $\vecs$ such that each $\sigma_i$ is conjugate to
$\sigma$, and in this case $\fe$ is isomorphic to the commuting variety of the corresponding 
simple Jordan algebra,  see Section~\ref{sect5}.

%{\small $4{\blacktriangleright}$} 
(4) \  The relationship with $\vecs$-commuting varieties implies that the multiplication map 
$\eus J\times\eus J\stackrel{\circ}{\to} \eus J$ is equidimensional if and only if $\eus J$ is the Jordan algebra of symmetric matrices.
The commuting variety of a simple Jordan algebra $\eus J$
is reducible, since $\eus J\times\{0\}$ and $\{0\}\times\eus J$ are always irreducible components; and there are certainly some other components.  

%{\small $5{\blacktriangleright}$}  
(5) \ Results stated in (2) rely on an interesting property of
$\BZ_2$-gradings. For any  $e\in\g_0$, its centraliser in $\g$ is also $\BZ_2$-graded:
$\g^e=\g^e_0\oplus\g^e_1$. Then we prove that \\
\centerline{$\dim\g^e_0+\rk\,\g \ge \dim\g^e_1$ }
and the equality occurs only if $e=0$ and $\sigma$ is of maximal rank. However, the proof of this inequality (Theorem~\ref{thm:strange-ineq}) is not quite uniform, and a better proof is welcome!
The required case-by-case calculations are lengthy and tedious, so that not all of them are actually presented, and a part of them is placed in Appendix~\ref{app:A}. We hope that an 
{\sl a priori\/} proof of this inequality might be related to a geometric property of centralisers of nilpotent elements in $\g_0$, see Conjecture~\ref{conj:a-la-alela}.

{\bf --}  \  Throughout, $G$ is a connected reductive algebraic group and $\g=\Lie(G)$. 
Then 
%{\bf --}  \  $\n_\g(\ah)$ (resp. 
$\z_\g(\ah)$ is the 
centraliser
of a subspace  $\ah\subset\g$,               
and 
the centraliser %in $\g$ 
of $x\in\g$ is denoted by $\z_\g(x)$ or $\g^x$.

{\bf --}  \  $\mathsf R(\lb)$ is a simple finite-dimensional $G$-module with highest weight $\lb$.

{\bf --}  \  Algebraic groups are denoted by capital Roman letters and their 
Lie algebras are denoted by the corresponding lower-case gothic letters.

%%%%%%%%%%%%%%%%   Section 1              
\section{Preliminaries on involutions and commuting varieties} 
\label{sect1}

\noindent
The set of all involutions of $\g$ is denoted by $\mathsf{Inv}(\g)$. 
The group of inner automorphisms  
$\mathsf{Int}(G)\simeq G/Z(G)$ acts on $\mathsf{Inv}(\g)$ by conjugation.
Two involutions are said to be {\it conjugate}, if they lie in the same $\mathsf{Int}(G)$-orbit.
If $\sigma\in\mathsf{Inv}(\g)$, then 
$\g=\g_0\oplus\g_1$ is the corresponding $\mathbb Z_2$-grading of
$\g$, where $\g_i=\{x\in\g\mid \sigma(x)=(-1)^ix\}$. We also say that 
$(\g,\g_0)$ is a {\it symmetric pair}.
Whenever we wish to stress that $\g_0$ and $\g_1$ 
are determined by $\sigma$, we write $\g^\sigma$ and $\g_1^{(\sigma)}$ for them.
We assume that $\sigma$ is induced by an involution of $G$, which is denoted by the same letter.
The connected subgroup of $G$ with Lie algebra $\g_0$ is denoted by $G_0$.
% while the fixed-point subgroup of $\sigma$ is denoted by $G^\sigma$. 
Hence $G_0$ is the identity component of $G^\sigma=\{g\in G\mid \sigma(g)=g\}$.
The representation of $G_0$ in $\g_1$ %, denoted $(G_0:\g_1)$, 
is the {\it isotropy representation\/}  of the symmetric space $G/G_0$.

We freely use invariant-theoretic results on the  $G_0$-action 
on $\g_1$ obtained in \cite{kr71}.  
A {\it Cartan subspace\/} (=\CSS) is a maximal subspace of $\g_1$ consisting of pairwise 
commuting semisimple elements. 
The Cartan subspaces are characterised by the following property: 
\begin{itemize}
\item[\refstepcounter{equation}(\theequation)\label{char-prop}]   
{\it Suppose that a subspace $\ah\subset \g_1$ consists of pairwise commuting
semisimple elements. Then $\ah$ is a \CSS\ if and only if $\z_\g(\ah)\cap\g_1=\ah$}
\ \cite[Ch.\,I]{kr71}.
\end{itemize}

\noindent
An element $x\in\g_1$ is called 
$G_0$-{\it regular\/} if the orbit $G_0{\cdot}x$ is of maximal dimension.
Let $\ce$ be a \CSS\ of $\g_1$. 
Below, we summarise some basic properties of the Cartan subspaces and isotropy representations:
\begin{itemize}
\item[\sf --]  All \CSS\ of $\g_1$ are $G_0$-conjugate and $G_0{\cdot}\ce$ is dense in $\g_1$;
\item[\sf --]  Every semisimple element of $\g_1$ is $G_0$-conjugate to an element of $\ce$;
\item[\sf --]  A semisimple element $x\in\g_1$ is $G_0$-regular  \ $\Leftrightarrow$ \ $\z_\g(x)\cap\g_1$ is a \CSS;
\item[\sf --]   The orbit $G_0{\cdot}x$ is closed if and only if $x$ is semisimple;
\item[\sf --]  The closure of $G_0{\cdot}x$ contains the origin if and only if $x$ is nilpotent;
\item[\sf --]  The number of nilpotent $G_0$-orbits in $\g_1$ is finite.
\end{itemize}

\textbullet\quad 
We say that $\sigma\in\mathsf{Inv}(\g)$ is 
{\it  of maximal rank\/} if $\g_1$ contains  a  Cartan subalgebra of $\g$.

\noindent
As is well known, (1) \ $\dim\g_1-\dim\g_0\le \rk\,\g$  for any $\sigma$, and the equality holds
if and only if $\sigma$ is of maximal rank;
(2) \ all involutions of maximal rank are conjugate;
(3) \ the involutions of maximal rank are inner {\sl if and only if\/} all exponents of $\g$ are odd.

\begin{lm}[\protect {\cite[Prop.\,5]{kr71}}]  \label{prop.5}
For any $x\in \g_1$, one has  $\dim\g_0-\dim\g^x_0=\dim\g_1-\dim\g^x_1$. Equivalently,
$\dim G{\cdot}x=2\dim G_0{\cdot}x$ \ for all $x\in \g_1$.
\end{lm}

Consequently, if $\sigma$ is of maximal rank, then 
\beq   \label{ravenstvo-max-rank}
     \dim\g^x_1=\dim\g^x_0+\rk\,\g .
\eeq
The property of having maximal rank is inheritable in the following sense.

\begin{lm}   \label{lm:inherit}
Let $\sigma$ be of maximal rank and $x\in\g_1$  semisimple.
Then the restriction of $\sigma$ to  $\g^x$ and 
$[\g^x,\g^x]$ is also of maximal rank.
\end{lm}

The {\it commuting variety\/} associated with $\sigma$ is 
\beq      \label{eq:sigma-com-var}
      \fe(\g_1)=\{(x,y)\in \g_1\times\g_1 \mid [x,y]=0\} .
\eeq
That is, $\fe(\g_1)$ is the zero-fibre of the commutator map $[\ ,\ ]_1:\g_1\times \g_1\to \g_0$. It is known that 
\begin{itemize}
\item \ $\ov{G_0{\cdot}(\ce\times\ce)}$ is always an irreducible component of 
$\fe(\g_1)$~\cite[Prop.\,3.7]{compos94};
\item \ if $\sigma$ is of maximal rank, then $\ov{G_0{\cdot}(\ce\times\ce)}=\fe(\g_1)$ and
$\g_1\times\g_1\to \g_0$ is equidimensional~\cite[Theorem\,3.2]{compos94}; moreover, all the fibres 
of $[\ ,\ ]_1$ are irreducible and normal~\cite[Cor.\,4.4]{compos94}.
\item \ $\fe(\g_1)$ can be reducible~\cite[Example\,3.5]{compos94}. 
\end{itemize}

\begin{ex}  \label{ex:usual-com-var}
Suppose that $\tilde\g=\g\oplus\g$ and $\sigma(x,y)=(y,x)$. Then $\tilde\g_0=\Delta(\g)$  and 
$\tilde\g_1=\{(x,-x)\mid x\in\g\}$. Here the commutator $\tilde\g_1\times\tilde\g_1\to \tilde\g_0$ coincides
with the usual commutator $\g\times\g\to\g$ and $\fe(\tilde\g_1)$ is isomorphic to the usual commuting variety of a semisimple Lie algebra $\g$. By a result of Richardson~\cite{ri79}, $\fe(\g)$ is irreducible and $\dim\fe(\g)=\dim\g+\rk \g$.
\end{ex}

A torus $S$ of $G$ is called $\sigma$-{\it anisotropic\/}, if $\sigma(s)=s^{-1}$ for all $s\in S$. 
%A \CSS\ in $\g_1$ is  the Lie algebra of a maximal $\sigma$-anisotropic  torus.
All maximal $\sigma$-anisotropic tori are $G_0$-conjugate, and if $C\subset G$ is a maximal $\sigma$-anisotropic torus, then  $\Lie(C)$ is a \CSS\ in 
$\g_{1}$.
Recall that a {\it restricted root\/} of $C$  is 
any non-trivial weight in the decomposition of $\g$ into the sum 
of weight spaces of $C$. Write $\Psi ^{C}(G/G_0)$ or just $\Psi (G/G_0)$ for the set of all restricted 
roots. Then
\beq   \label{eq:restr-root}
   \g=\g^C\oplus\bigl(\bigoplus_{\gamma\in\Psi(G/G_0)}\g_\gamma \bigr) .
\eeq
We use the additive notation for the operation in $\mathfrak X(C)$, the character 
group of $C$, and regard $\Psi(G/G_0)$ as a subset of the vector space
$\mathfrak X(C)\otimes_\BZ\BR$.
The set 
$\Psi(G/G_0)$ satisfies the usual axioms of finite root systems \cite{helg}. The notable difference from 
the structure theory of split semisimple Lie algebras is that the root system 
$\Psi(G/G_0)$ can be non-reduced and that multiplicities $m_\gamma=\dim \g_\gamma$
($\gamma\in\Psi(G/G_0)$) can be greater than $1$.

For all involutions of simple Lie algebras, the restricted root systems and the respective multiplicities are known, see \cite[Ch.\,X, Table\,VI]{helg}.

%%%%%%%%%%%%%%%%   Section 1.5              
\section{Commuting involutions and quaternionic decompositions}
\label{sect1.5}

Let $\sigma_1$ and $\sigma_2$ be different commuting involutions of $\g$. 
The corresponding $\mathbb Z_2\times\mathbb Z_2$-grading of $\g$ is:
\beq  \label{eq:quatern_summa}
\g=\bigoplus_{i,j=0,1}\g_{ij}, \ \text{ where }\ \g_{ij}=\{x\in\g\mid   \sigma_1(x)=(-1)^ix \ \ \& \ \ 
\sigma_2(x)=(-1)^jx\}.
\eeq
We also say that it is a {\it quaternionic decomposition\/} of $\g$ 
(determined by $\sigma_1$ and $\sigma_2$). Set $\sigma_3:=\sigma_1\sigma_2$ and
$\vecs=(\sigma_1,\sigma_2,\sigma_3)$.
The pairwise commuting involutions $\sigma_1,\sigma_2$, and $\sigma_3$ are said to be {\it big}.
The induced involutions on the fixed-point subalgebras 
$\g^{\sigma_1},\g^{\sigma_2},\g^{\sigma_3}$
are said to be {\it little}. The same terminology applies to the corresponding $\BZ_2$-gradings, isotropy representations, and \CSS.
Thus, associated with \eqref{eq:quatern_summa}, one has three big 
and three little $\BZ_2$-gradings.
It is convenient for us to organise the summands of  \eqref{eq:quatern_summa}
in a $2\times 2$ ``matrix'':

\hbox to \textwidth{\enspace \refstepcounter{equation} (\theequation) 
\hfil   \label{eq:quatern_matrix}
%\begin{figure}[h]  
{\setlength{\unitlength}{0.02in}
%\begin{center}
\begin{picture}(35,27)(0,5)
\put(-12,7){$\g=$}
    \put(8,15){$\g_{00}$}    \put(28,15){$\g_{01}$}
    \put(8,3){$\g_{10}$}      \put(28,3){$\g_{11}$}
\qbezier[20](5,9),(22,9),(40,9)            % horisontal line
\qbezier[20](23,-1),(23,11),(23,24)      % vertical line
\put(19.9,7){$\oplus$}   %   central plus
\put(21,-7){{\color{my_color}$\sigma_2$}}
\put(43,7){{\color{my_color}$\sigma_1$}}
\end{picture}  \hfil
}}
\vskip2.5ex

\noindent
Here the horizontal (resp. vertical) dotted line separates the eigenspaces of  $\sigma_1$ 
(resp. $\sigma_2$), whereas two diagonals of this matrix represent the eigenspaces
of $\sigma_3$.
Hence the first row,  first column, and the main diagonal
represent the three little $\BZ_2$-gradings  (of $\g^{\sigma_1}$,
$\g^{\sigma_2}$, and
$\g^{\sigma_3}$, respectively).

We repeatedly use the following notation for the eigenspaces of $\sigma_1$ and 
$\sigma_2$:

$\g^{\sigma_1}=\g_{0\star}:=\g_{00}\oplus\g_{01}$, \ $\g_{1\star}:=\g_{10}\oplus\g_{11}$, \qquad %etc.
$\g^{\sigma_2}=\g_{\star 0}:=\g_{00}\oplus\g_{10}$, \ $\g_{\star 1}:=\g_{01}\oplus\g_{11}$.
\vskip.6ex\noindent
Likewise, $G_{0\star}$ (resp. $G_{\star 0}$) is the connected subgroup of $G$ corresponding to $\g_{0\star}$ (resp. $\g_{\star 0}$), 
$G_{00}$ is the connected subgroup of $G$ corresponding to $\g_{00}$, etc. 
If $\q$ is a $\vecs$-stable subalgebra of $\g$, then
$\q=\bigoplus_{i,j} \q_{ij}$ is the induced quaternionic decomposition of $\q$, and  $Q, Q_{00}$ are the corresponding connected subgroups.

Following Vinberg \cite[0.3]{vinb}, we say that a triple $\{\sigma_1, \sigma_2, \sigma_3\}
\subset\Inv(\g)$ is a {\it triad\/} if all three involutions are conjugate and 
$\sigma_1\sigma_2=\sigma_3$. A complete classification of triads is obtained in 
\cite[Sect.\,3]{vinb}. The triads lead to the ``most symmetric'' quaternionic  decompositions.
In \cite{comm-inv1}, we considered less restrictive conditions on the $\sigma_i$'s.
We say that  $\{\sigma_1, \sigma_2\}\subset \Inv(\g)$
is a {\it dyad\/} if $\sigma_1, \sigma_2$
are conjugate and $\sigma_1\sigma_2=\sigma_2\sigma_1$ (no conditions on $\sigma_3$!).

The product of two conjugate involutions (not necessarily commuting) is always an
inner automorphism of $\g$. 
For, if $\sigma_2=\Int(g){\cdot}\sigma_1{\cdot}\Int(g^{-1})$, then
$\sigma_1\sigma_2=\Int(\sigma_1(g)g^{-1})$.
Therefore, any triad consists of
inner involutions (but not any inner involution gives rise to a triad!). 
However, any involution can be a member of a dyad  \cite[Prop.\,2.4]{comm-inv1}.
But the third involution, $\sigma_3$, is then necessarily inner.
 
\begin{prop}[\protect see {\cite[Prop.\,2.2(1)]{comm-inv1}}]    \label{prop:povtor}
Suppose that $\mu\in\Inv(\g)$ is  inner. 
Then there are  commuting
involutions of maximal rank, $\sigma_1$ and $\sigma_2$,  such that  $\mu=\sigma_1\sigma_2$.
Moreover,  $\sigma_1$ and $\sigma_2$ induce an involution of maximal rank of
$\g^{\mu}$.
\end{prop}

 For $(ij)\ne (00)$, let $\ce_{ij}$ be a \CSS\ of $\g_{ij}$; that is, 
a little  \CSS\ related to the little $\BZ_2$-grading 
$\g_{00}\oplus\g_{ij}$. %Let us say that $\ce_{ij}$ is a {\it little\/} \CSS.
There are also big \CSS\ in the $(-1)$-eigenspaces of three big involutions:
\vskip.7ex
\centerline{
$\ce_{1\star}\subset \g_{1\star}$, \ 
$\ce_{\star 1}\subset \g_{\star 1}$, \ 
$\ce_{\star,1-\star}\subset \g_{\star,1-\star}:=\g_{01}\oplus\g_{10}$.
}
\vskip.8ex\noindent
Each little \CSS\ can be included in two big \CSS. E.g., because $\g_{10}\subset
\g_{1\star}$ and $\g_{10}\subset\g_{\star,1-\star}$, one can choose Cartan subspaces
$\ce_{1\star}$ and 
$\ce_{\star,1-\star}$ such that 
$\ce_{10}\subset \ce_{1\star}$ and $\ce_{10}\subset \ce_{\star,1-\star}$.
If at least one equality occurs among all such inclusions, then this will be referred to as a {\it
coincidence\/} of \CSS\ (for a given quaternionic decomposition).

In \cite{comm-inv1}, we obtained two sufficient conditions for a coincidence of \CSS:

\begin{thm}[\protect see {\cite[Thm.\,3.3 \& 3.7]{comm-inv1}}]   \label{thm:povtor}
\leavevmode\par
\begin{enumerate}
\item \  Suppose that $\sigma_1$ is of maximal rank. Then 
\begin{itemize}
\item \  any little  \CSS\ \ $\ce_{11}\subset \g_{11}$
is also a \CSS\ \ in $\g_{\star 1}$, i.e., for $\sigma_2$; 
\item \  any little \CSS\ \ $\ce_{10}\subset \g_{10}$
is also a \CSS\ \ in $\g_{10}\oplus\g_{01}$, i.e., for $\sigma_3$.
\end{itemize}
\item \  Suppose that $\{\sigma_1,\sigma_2\}$ is a dyad. 
Then any little \CSS\ \ $\ce_{11}\subset \g_{11}$
is also a \CSS\ \ in $\g_{1\star}$ or $\g_{\star 1}$, i.e., for $\sigma_1$ or $\sigma_2$.
\end{enumerate}
\end{thm}

The coincidences of \CSS\ in Theorem~\ref{thm:povtor}(2) can formally be expressed as 
$\ce_{11}=\ce_{1\star}$ or $\ce_{11}=\ce_{\star 1}$, and likewise in all other possible cases. 
In view of~\eqref{char-prop}, any coincidence of \CSS\ can be restated as certain property of the
little \CSS\ in question. For instance, 
the first coincidence in Theorem~\ref{thm:povtor}(1) means that if $x\in \g_{11}$ is a generic
semisimple element (i.e., $x$ belong to a unique little \CSS), then $\z_\g(x)_{\star 1}=\z_\g(x)_{11}=\ce_{11}$, and hence $\z_\g(x)_{01}=0$.

%%%%%%%%%%%%%%%%   Section 2.
\section{Commuting varieties and homogeneous Cartan subspaces}  
\label{sect2}

\noindent
Consider a quaternionic decomposition~\eqref{eq:quatern_matrix}.
For any permutation $(\ap,\beta,\gamma)$ of the set $\{01,10,11\}$, there is
the commutator mapping $\vp_{\ap,\beta}^\gamma: \g_\ap\times \g_\beta\to \g_\gamma$.
Clearly, $\vp_{\ap,\beta}^\gamma$ is $G_{00}$-equivariant.
As our main interest is in fibres of this mapping, we do not distinguish 
$\vp_{\ap,\beta}^\gamma$ and $\vp_{\beta,\ap}^\gamma$.
We concentrate on the following problems:
\begin{itemize}
\item \quad When is $\vp_{\ap,\beta}^\gamma$ dominant?
\item \quad What is the dimension of $(\vp_{\ap,\beta}^\gamma)^{-1}(0)$ ?
\item \quad How to describe the irreducible components of $(\vp_{\ap,\beta}^\gamma)^{-1}(0)$ ?
\item \quad  When is $\vp_{\ap,\beta}^\gamma$ equidimensional?
\end{itemize}
The variety $\fe_{\ap,\beta}^\gamma=(\vp_{\ap,\beta}^\gamma)^{-1}(0)$ is said to be a 
$\vecs$-{\it  commuting variety}. For general quaternionic decompositions, one has  three 
such varieties, and their properties can be rather different.
We mainly restrict ourselves with considering the test case:
\beq \label{eq:comm-map}
   \vp=\vp_{10,11}^{01}: \g_{10}\times \g_{11}\to \g_{01} .
\eeq
and also write $\fe$ in place of $\fe_{10,11}^{01}$. Note that we can regard $\vp$  as a quadratic 
map from $\g_{1\star}$ to $\g_{01}$, and $\fe$ as 
subvariety of $\g_{1\star}$. The following example shows that the commuting variety in
\eqref{eq:sigma-com-var} is a particular case of this construction.

\begin{ex}  \label{ex:g=s+s}
Let $\g$ be a reductive Lie algebra and $\sigma$ an involution of
$\g$ with the corresponding $\BZ_2$-grading  $\g=\g_0\oplus\g_1$.
%We specify the requirements on $\sigma$ below.
Set $\tilde\g=\g\oplus\g$  and define three involutions of $\tilde\g$ as follows:

\centerline{$\sigma_1(x_1,x_2)=(\sigma(x_1),\sigma(x_2))$, \quad
$\sigma_2(x_1,x_2)=(x_2,x_1)$, \quad
$\sigma_3=\sigma_1\sigma_2$.}

\noindent
Then $\tilde\g^{\sigma_1}=\g_0\oplus\g_0$; \quad 
$\tilde\g^{\sigma_2}=\Delta(\g)$, the diagonal in $\g\oplus\g$; \quad 
$\tilde\g^{\sigma_3}=\{ (x,\sigma(x))\mid x\in\g \}$. %=:\Delta_\vartheta(\h)$.
\\
Set $\Delta_-(M):=\{(m,-m)\mid m\in M\}$ for any subspace $M\subset \g$.
Then the  corresponding quaternionic decomposition is: 
\begin{center}
\setlength{\unitlength}{0.02in}
\begin{picture}(90,28)(-10,2)
\put(-25,7){$\tilde\g=$}
    \put(3,15){$\Delta(\g_0)$}    \put(40,15){$\Delta_-(\g_0)$}
    \put(3,0){$\Delta(\g_1)$}      \put(40,0){$\Delta_-(\g_1)$}
\qbezier[40](-5,9),(30,9),(70,9)            % horisontal line
\qbezier[20](33,-3),(33,11),(33,24)      % vertical line
\put(29.9,7){$\oplus$}   %   central plus
\put(31,-11){{\color{my_color}$\sigma_2$}}
\put(80,7){{\color{my_color}$\sigma_1$}}
\end{picture}  
\end{center}
\vskip3ex
Upon the obvious identifications $\Delta(\g_1)\simeq \Delta_-(\g_1)\simeq \g_1$, etc., our 
test commutator
map $\tilde\g_{10}\times \tilde\g_{11}\to \tilde\g_{01}$ becomes the commutator 
$\g_1\times\g_1\to \g_0$ associated with $\sigma\in \mathsf{Inv}(\g)$; whereas two other 
commutator maps are identified with the bracket $\g_0\times \g_1\to \g_1$. Therefore, the concept of a $\vecs$-commuting variety provides a uniform setting for studying the fibres of 
both $\g_1\times\g_1\to \g_0$ and $\g_0\times \g_1\to \g_1$.
\end{ex}

\begin{lm}        \label{lm:dominant}
Commutator map \eqref{eq:comm-map}
 is dominant if and only if there exist $x\in\g_{10}$ and $y\in\g_{11}$ such that 
 $\z_\g(x)_{01}\cap\z_\g(y)_{01}=\{0\}$.
\end{lm}
\begin{proof}
A morphism of irreducible varieties is dominant if and only if its differential at some point is onto.
As $\vp$ is bilinear, an easy computation shows that  
$\textsl{d}\vp_{(x,y)}(\xi,\eta)=[x,\eta]+[\xi,y]$, $\xi\in\g_{10}$, $\eta\in\g_{11}$. Hence
$\Ima\textsl{d}\vp_{(x,y)}=[\g_{11},x]+[\g_{10},y]$, and taking the orthogonal complement with 
respect to the restriction of the Killing form to $\g_{01}$ yields
$(\Ima\textsl{d}\vp_{(x,y)})^\perp =\z_\g(x)_{01}\cap\z_\g(y)_{01}$.
\end{proof}

As we see below, certain \CSS\ in $\g_{1\star}$ play an important  role in describing irreducible
components of $\fe$.

\begin{df}
A big Cartan subspace $\ce_{1\star}\subset\g_{1\star}$ is said to be {\it homogeneous\/} if
it is $\sigma_2$-stable (or, equivalently,  $\sigma_3$-stable). In other words, if one has 
$\ce_{1\star}=\ah_{10}\oplus \ah_{11}$ with $\ah_{1j}\subset \g_{1j}$.
\end{df}

{\it Remark.} A coincidence of \CSS\ means that there is a homogeneous \CSS\ of special form.
For instance,  if $\ce_{11}=\ce_{1\star}$, then $\ce_{11}$ is a 
homogeneous \CSS\ in $\g_{1\star}$, with trivial $\g_{10}$-component.
\begin{lm}
{\textsf (1)} \ Homogeneous \CSS\ always exist. \\
{\textsf (2)} \ Moreover, if $x\in \g_{10}, y\in\g_{11}$ are commuting semisimple elements, then there exists a homogeneous \CSS\ in $\g_{1\star}$ containing both of them.
\end{lm}
\begin{proof}
1) Take a little \CSS\ $\ce_{10}$ and consider the $\vecs$-stable reductive
subalgebra $\z_\g(\ce_{10})$. If $\tilde\ah_{11}$ is a little \CSS\ in $\z_\g(\ce_{10})_{11}$, then
$\ce_{10}\oplus\tilde\ah_{11}$ is a homogeneous \CSS\ in $\g_{1\star}$.
\\
2) Consider the $\vecs$-stable reductive subalgebra  $\el=\z_\g(x)\cap\z_\g(y)$.
By the previous part, there exists a homogeneous \CSS\ in $\el_{1\star}$, say $\tilde\ce_{1\star}$.
Since $x,y$ are central in $\el$, we have $x,y\in \tilde\ce_{1\star}$.
It is also clear that $\tilde\ce_{1\star}$ is a \CSS\ in $\g_{1\star}$.
\end{proof}
If  $\ce_{1\star}=\ah_{10}\oplus \ah_{11}$ is a homogeneous \CSS, then $[\ah_{01}, \ah_{11}]=0$ 
and hence $\ov{G_{00}{\cdot}\ce_{1\star}}\subset \fe$. However, a stronger result is true.

\begin{thm}    \label{thm:irr-comp-Fe}
\leavevmode\par
\begin{itemize}
\item[\sf(i)] \ Let $\ce_{1\star}$ be a homogeneous \CSS\ in $\g_{1\star}$. Then
$\ov{G_{00}{\cdot}\ce_{1\star}}\subset \fe$ is an irreducible component of $\fe$.
\item[\sf(ii)] \ If two homogeneous \CSS\ in $\g_{1\star}$ are not $G_{00}$-conjugate, then the corresponding irreducible components are different.
\end{itemize}
\end{thm}
\begin{proof}
(i) The centraliser of $\ce_{1\star}$ is $\vecs$-stable. Hence
$\z_{\g}(\ce_{1\star})=\displaystyle \bigoplus_{i,j=0,1} \ah_{ij}$, and here $\ce_{1\star}=\ah_{10}\oplus\ah_{11}$. Recall that 
$\ov{G_{0\star}{\cdot}\ce_{1\star}}=\g_{1\star}$. Therefore,
%\[
   $\dim\ce_{1\star}+\dim G_{0\star}-\dim\ah_{00}-\dim\ah_{01}=\dim \g_{1\star}$.
%\]
It follows that 
\beq   \label{eq:dim-razneseniya}
  \dim \ov{G_{00}{\cdot}\ce_{1\star}}=\dim\ce_{1\star}+ \dim G_{00} -\dim\ah_{00}=\dim\g_{1\star}
  -\dim\g_{01}+\dim\ah_{01} .
\eeq
On the other hand, let $x+y\in\ce_{1\star}$ ($x\in\g_{10}, y\in\g_{11}$).
The proof of Lemma~\ref{lm:dominant} shows that  
$\dim(\Ima \textsl{d}\vp_{(x,y)})=\dim\g_{01}-\dim (\z_\g(x)_{01}\cap\z_\g(y)_{01})$.
Now, if $x+y\in\ce_{1\star}$ is generic, then $\z_\g(x)\cap\z_\g(y)=\z_\g(x+y)=\z_\g(\ce_{1\star})$.
Hence $\dim(\Ima \textsl{d}\vp_{(x,y)})=\dim\g_{01}-\dim \ah_{01}$. This means that any irreducible
component of $\fe$ containing $(x,y)$ has dimension at most 
\[
 \dim\g_{1\star}-\dim(\Ima \textsl{d}\vp_{(x,y)}) =\dim\g_{1\star}-\dim\g_{01}+\dim \ah_{01} .
\] 
Comparing with \eqref{eq:dim-razneseniya} shows that 
$\ov{G_{00}{\cdot}\ce_{1\star}}$ is an irreducible component of $\fe$ containing $(x,y)$,
and $(x,y)$ is a smooth point of $\ov{G_{00}{\cdot}\ce_{1\star}}$.

(ii) As we have just shown, if $(x,y)\in \ce_{1\star}$ is generic, then it belongs to a unique irreducible component of $\fe$ (and to a unique \CSS\ in $\g_{1\star}$).
\end{proof}

\begin{utv}  \label{thm:finitely-odnor-CSS}
 The number of $G_{00}$-orbits of homogeneous \CSS\ in $\g_{1\star}$ is finite.
\end{utv}
\noindent
{\it First proof.} \ Since the number of irreducible components is finite, this readily follows from
Theorem~\ref{thm:irr-comp-Fe}. However, this can also be proved in a 
different way. As the second proof has its own merits,  we provide it below.
\begin{proof}[Second proof]
Recall that $G_{00}\subset G_{0\star}$ are connected reductive groups and
\un{all} big \CSS\ in $\g_{1\star}$ form  a single $G_{0\star}$-orbit. 
Let $\ce_{1\star}$ be a homogeneous \CSS. Set 
\begin{gather*}
   N=\{g\in G_{0\star} \mid g{\cdot}\ce_{1\star}=\ce_{1\star} \} ,  \\
   \MM=\{g\in G_{0\star} \mid g{\cdot}\ce_{1\star} \text{ is homogeneous }\} .
\end{gather*}
Note that $N$ is reductive, but not connected, since $N$ is mapped onto the (finite) little Weyl group associated with $\ce_{1\star}$. 
If $g\in \MM$, $s\in G_{00}$, and $z\in N$, then $sgz\in \MM$. Therefore, 
$\MM$ is a union of  $(G_{00},N)$-cosets, and our task is to prove that
$G_{00}\backslash \MM/N$ is finite.

If $g\in\MM$, then $g{\cdot}\ce_{1\star}=\sigma_2(g)\ce_{1\star}$. Hence 
$g^{-1}\sigma_2(g)\in N$. Since $G_{00}\subset G^{\sigma_2}$, the map 
\[
  \psi_\MM:G_{00}\backslash\MM\to N,  \quad G_{00}g\mapsto g^{-1}\sigma_2(g)
\] 
is well-defined. Note that $N$ is $\sigma_2$-stable and 
the range of $\psi_\MM$ belongs to the closed subset 
\[
   \mathcal Q=\mathcal Q(N)=\{g\in N\mid \sigma_2(g)=g^{-1}\} .
\]
The twisted $N$-action on $N$ is defined by $z\star x=zx\sigma_2(z)^{-1}$. Obviously,
$\mathcal Q$ is stable under the twisted action of $N$.
Moreover, $\psi_\MM(gz)=z^{-1}\psi_\MM(g)\sigma_2(z)$. Hence
$\Ima(\vp_\MM)\subset \mathcal Q$ is the union of twisted $N$-orbits, and each twisted $N$-orbit gives rise to a $G_{00}$-orbit of homogeneous \CSS.
%That is, the assertion of the theorem is equivalent to that $\Ima(\vp_\MM)$ consists of finitely
%many twisted conjugacy classes.
It follows from~\cite[Sect.\,9]{ri82} that $\mathcal Q$ is a finite union of twisted $N$-orbits, which 
is sufficient for our purpose. (See also remark below.)
\end{proof}

\begin{rmk}
Richardson's results on twisted orbits \cite[Sect.\,9]{ri82}, specifically Proposition~9.1, are 
stated for a {\sl connected\/} reductive group $G$, whereas we apply them to the reductive 
non-connected group $N$ (in place of $G$). But his argument can easily be adjusted to cover 
the case ofvnon-connected reductive groups.
That is, one can give a version of Richardson's Proposition~9.1 for non-connected groups $G$.
\end{rmk}

\begin{df}
For a homogeneous \CSS\ $\ce_{1\star}\subset \g_{1\star}$, the irreducible component 
$\ov{G_{00}{\cdot}\ce_{1\star}}\subset \fe$ is said to be {\it standard}.
\end{df}
Since all big \CSS\ in $\g_{1\star}$ are $G_{0\star}$-conjugate, their centralisers in $\g_{0\star}$ are 
essentially ``the same''. The centraliser in $\g_{0\star}$ of a homogeneous \CSS\ splits, and these
splittings can be quite different. That is,  $\dim\z_\g(\ce_{1\star})_{01}$  can be different
for different homogeneous \CSS, and this leads to a new phenomenon that 
standard irreducible components of $\fe$ may have different dimensions, 
cf.~\eqref{eq:dim-razneseniya}.
Moreover, there can also be some ``non-standard'' irreducible components of $\fe$ 
that contain no semisimple elements at all.

By Theorem~\ref{thm:irr-comp-Fe}, a necessary condition for $\fe$ to be irreducible is that 
all homogeneous \CSS\ in $\g_{1\star}$ are $G_{00}$-conjugate, i.e.,
there is only one standard component. 
%This seems to be a rare phenomenon, although not impossible. 
If $\ce_{1\star}=\ah_{10}\oplus\ah_{11}$ is a homogeneous \CSS\ with
$\dim\ah_{1i}=d_i$, then $(d_0,d_1)$ is called the {\it dimension vector}. 
Obviously, two homogeneous \CSS\ with different dimension vectors are not $G_{00}$-conjugate.

\begin{thm}    \label{thm:krit-odno-homog-CSS}
1)  If $\ce_{1\star}=\ah_{10}\oplus\ah_{11}$ is a homogeneous \CSS\ with  dimension vector $(d_0,d_1)$, then
$d_0\le\dim\ce_{10}$ and $d_1\le\dim\ce_{11}$; hence $\dim\ce_{1\star}\le \dim\ce_{10}+
\dim\ce_{11}$.

2) All homogeneous \CSS\ in $\g_{1\star}$ are $G_{00}$-conjugate if and only if\/ 
$\dim\ce_{1\star}= \dim\ce_{10}+\dim\ce_{11}$.
\end{thm}
\begin{proof}
1) Being a toral subalgebra of $\g_{1j}$, $\ah_{1j}$ is contained in a little \CSS\ in $\g_{1j}$.

2)  ``if'' part:  Let $\ce_{1\star}$ and $\tilde\ce_{1\star}=\tilde\ah_{10}\oplus\tilde\ah_{11}$ be two
homogeneous \CSS. By part 1), $\dim\ah_{01}=\dim\tilde\ah_{01}=\dim\ce_{10}$. Therefore,
both $\ah_{01}$ and $\tilde\ah_{01}$ are little \CSS,
%. Hence $\ah_{01}$ and $\tilde\ah_{01}$ 
they are $G_{00}$-conjugate, and
we may assume 
that $\ah_{01}=\tilde\ah_{01}$. Consider then the $\vecs$-stable reductive algebra
$\z_\g(\ah_{10})$. As $\ah_{10}$ is a central toral subalgebra,  
$\z_\g(\ah_{10})=\ah_{10}\oplus\es$, where $\es$ is reductive and $\vecs$-stable.
By construction, $\es_{10}=\{0\}$ and 
$\ah_{11}, \tilde\ah_{11}\subset \es_{11}$. Moreover, these are little \CSS\ in $\es_{11}$
(otherwise,  $\ce_{1\star}$ or $\tilde\ce_{1\star}$ wouldn't be maximal). Therefore, 
$\ah_{01}$ and $\tilde\ah_{01}$ are $S_{00}$-conjugate, which implies that 
$\ce_{1\star}$ or $\tilde\ce_{1\star}$ are $G_{00}$-conjugate.

``only if'' part: Assuming that $\dim\ce_{1\star} < \dim\ce_{10}+\dim\ce_{11}$, we construct 
two homogeneous \CSS\ with different dimension vectors.
First, let us take a little \CSS\ $\ce_{10}$ and choose a little \CSS\ in $\z_\g(\ce_{10})_{11}$,
say $\tilde\ah_{11}$. This yields a homogeneous \CSS\ with dimension vector
$(\dim\ce_{10}, \dim\ce_{1\star}-\dim\ce_{10})$. On the other hand, one can start with a little
\CSS\ $\ce_{11}$, etc., which yields a homogeneous \CSS\ with dimension vector
$(\dim\ce_{1\star}-\dim\ce_{11}, \dim\ce_{11} )$.
\end{proof}

Note that $\dim\ce_{ij}>0$ whenever $\g_{ij}\ne \{0\}$. Therefore,  a coincidence of
\CSS\ of the form $\ce_{11}=\ce_{1\star}$ or $\ce_{10}=\ce_{1\star}$ certainly excludes the possibility to have a unique standard component of $\fe$. For our test 
commutator~\eqref{eq:comm-map}, one may envisage several samples of good behaviour (not necessarily altogether):

(1)  All irreducible components of $\fe$ are standard (possibly of different dimension);

(2) $\vp$ is surjective and equidimensional; hence, flat;

(3)  $\fe$ has a unique standard component, but also may be some other components.
\\[.6ex]
Property~(3) always holds in the setting of Example~\ref{ex:g=s+s}, with any $\sigma$;
and for $\sigma$ of maximal rank, one gets a rare situation, where all three
properties are satisfied. All quaternionic decompositions of simple Lie algebras can be written out
explicitly, and then the presence of~(3) amounts to a routine verification of the equality in
Theorem~\ref{thm:krit-odno-homog-CSS}(2).

\begin{prop}    \label{lm:dyad-onto}
 Let $\{\sigma_1,\sigma_2\}$ be a dyad. Then $\dim\g_{10}=\dim\g_{01}$ and
 $\vp: \g_{10}\times\g_{11}\to \g_{01}$ is onto. (Therefore, 
 $\dim\vp^{-1}(\xi)\ge \dim\g_{11}$ \ for all $\xi\in \g_{01}$.) Moreover, $\{0\}\times \g_{11}$ is a standard irreducible component of\/ $\fe$ of minimal dimension.
\end{prop}
\begin{proof}
Since $\dim\g^{\sigma_1}=\dim\g^{\sigma_2}$, we have $\dim\g_{10}=\dim\g_{01}$.
By Theorem~\ref{thm:povtor}(2), 
any little \CSS\ $\ce_{11}\subset \g_{11}$ is also a big \CSS\ in 
$\g_{1\star}$. Therefore, $\ce_{11}$ is a homogeneous \CSS\ and $\ov{G_{00}{\cdot}\ce_{11}}=
\g_{11}$ is an irreducible component of $\fe$. Furthermore, if $x\in \ce_{11}$ is generic, then 
$\z_\g(x)\cap\g_{1\star}=\ce_{11}$, i.e., $\z_\g(x)\cap\g_{10}=\{0\}$. 
Therefore, $\dim [\g_{10},x]=\dim\g_{10}$, i.e., $[\g_{10},x]=\g_{01}$.
\end{proof}

 %%%%%%%%%%%%%%%%   Section 3
\section{Dyads of maximal rank and commuting varieties}  
\label{sect3}

\noindent
Let $\{\sigma_1,\sigma_2\}$ be a dyad of maximal rank, i.e., both $\sigma_1,\sigma_2$ are of maximal rank. Recall that this implies that
$\sigma_3=\sigma_1\sigma_2$ is inner,  
$\dim\g_{01}=\dim \g_{10}$, and, by Prop.~\ref{prop:povtor}, $\g^{\sigma_3}=\g_{00}\oplus \g_{11}$ is a $\BZ_2$-grading of maximal rank. In particular, $\g_{11}$ contains a \CSA\ of $\g$ 
and any  \CSS\ in $\g_{1\star}$ or $\g_{\star 1}$ is a \CSA.
The main result of this section is

\begin{thm}   \label{thm:main-EQ}
Let  $\{\sigma_1,\sigma_2\}$ be a dyad of maximal rank. Then 
\begin{itemize}
\item[\sf (i)] \  
the commutator mapping 
    $ \vp:\g_{10}\times \g_{11}\to \g_{01}$
is surjective and equidimensional;  
\item[\sf (ii)] \   each irreducible component of $\fe=\vp^{-1}(0)$ is standard,
i.e., is the closure of the $G_{00}$-saturation of a homogeneous \CSS\ in $\g_{1\star}$;
\item[\sf (iii)] \   the ideal of $\fe$ is generated by quadrics $\vp^{\#}(\g_{01}^*)$, where 
$\vp^{\#}: \bbk[\g_{01}]\to \bbk[\g_{10}]\otimes\bbk[\g_{11}]$ is the comorphism. (That is, the scheme $\vp^{-1}(0)$ is a reduced complete intersection).
\end{itemize}
\end{thm}
\begin{proof}
If $\q$ is a $\vecs$-stable reductive subalgebra of $\g$, then $\fe_\q$ stands for 
the zero-fibre of the commutator $\q_{10}\times\q_{11}\to \q_{01}$. Clearly, 
$\fe_q\subset \fe=\fe_\g$.
Since $\sigma_1$ and $\sigma_2$ are of maximal rank, the centre of $\g$, $\z(\g)$, 
is contained in $\g_{11}$. Consequently, $\fe_\g\simeq \fe_{[\g,\g]}\times \z(\g)$ and without loss
of generality, we may assume that $\g$ is semisimple.

By Proposition~\ref{lm:dyad-onto}, $\vp$ is onto and  $\dim\fe\ge\dim\g_{11}$. In this situation, $\vp$ is equidimensional if and only if $\dim\fe=\dim\g_{11}$.
If $\ce_{1\star}$ is a
homogeneous \CSS, then it is necessarily a \CSA\ of $\g$. That is, $\z_\g(\ce_{1\star})_{01}=0$
for \un{all} homogeneous \CSS\ and $\dim \ov{G_{00}{\cdot}\ce_{1\star}}=\dim \g_{11}$. 
Hence all the standard components of $\fe$ have the same (expected) dimension, and for 
(i) and (ii) it suffices to prove that there is no other irreducible components.

To this end, we argue by induction on $\rk\,\g=\dim\ce_{11}$.

{\bf --} \ If $\dim\ce_{11}=1$, then $\g=\tri$ and the assertion is true.

{\bf --} \ Suppose that $\rk\,\g>1$ and the assertion holds for all dyads of maximal rank for 
semisimple algebras of rank smaller than $\rk\,\g$. 

\textsf{1)} \  Take $(x,y)\in\fe$, $y\in\g_{11}$, and let $y=y_s+y_n$ be the Jordan decomposition. 
Then $[x,y_s]=0$. If $y_s\ne 0$, then
$y_n\in \es:= [\z_\g(y_s),\z_\g(y_s)]$ and $\rk\,\es<\rk\,\g$. By Lemma~\ref{lm:inherit},
$\sigma_i\vert_\es$, $i=1,2$, are again involutions of maximal rank. Let $\z$ denote the centre of
$\z_\g(y_s)$, so that $\z_\g(y_s)=\z\oplus\es$ and $y_s\in\z$. Since both $\sigma_1$ and $\sigma_2$ are 
of maximal rank, $\z\subset \g_{11}$ and hence $x\in \es$.
By the induction assumption, $(x,y_n)\in \es_{10}\oplus\es_{11}$ lies in a standard irreducible
component of $\fe_\es$. Obviously, adding a central summand does not affect this property, 
hence $(x,y)$ lies in a standard component of $\fe_{\z_\g(y_s)}$. As $\rk\,\z_\g(y_s)=\rk\,\g$, this also means that $(x,y)$ lies in a standard component of $\fe$.

\textsf{2)} \  Hence it suffices to consider the case in which $y=y_n$. Write $\N_{11}$ for the 
closed set of all nilpotent elements in $\g_{11}$.
Let $\ck$ be an irreducible component of $\fe$, hence $\dim \ck\ge\dim\g_{11}$. Then 
$\ck_1:=\ck\cap (\g_{10}\times\N_{11})$ is a closed subvariety of $\ck$. If $\ck_1\ne \ck$, then, 
by  part 1), all the points in $\ck\setminus \ck_1$ belong to standard irreducible components. 
Consequently, $\ck$ must be one of the standard components.

\textsf{3)} \  The next possibility is that $\ck=\ck_1$.
%Assume that $C=C_1$ is a ``bad'' component, and 
Let $p:\g_{10}\times\g_{11}\to \g_{11}$ be the projection.
Then $p(\ck)\subset \N_{11}$, and therefore $\ov{p(\ck)}=\ov{G_{00}{\cdot}y}$
is the closure of a nilpotent $G_{00}$-orbit. 

If $y=0$, then $\ck=\g_{10}\times\{0\}$. Let $\ce_{10}$ be a little \CSS.
The fact that $\ov{G_{00}{\cdot}(\ce_{10}\times\{0\})}=\g_{10}\times\{0\}$ is an irreducible component
of $\fe$ implies that $\z_\g(\ce_{10})_{11}=\{0\}$, whence $\ce_{10}$ is also a \CSS\ in $\g_{1\star}$.
That is, $\ce_{10}$ is a \CSA\ of $\g$. 
(Incidentally, this means that the (-1)-eigenspace of $\sigma_3$ contains a \CSA, i.e., $\{\sigma_1,\sigma_2,\sigma_3\}$ is actually a triad.) Anyway, we see that if $y=0$,
then such $\ck$ appears to be a standard component.
 
\textsf{4)} \   Finally, we prove that the case in which $\ck=\ck_1$ and $y\ne 0$ is impossible. 
Assuming the contrary,  we would have
 \begin{multline*}
   \dim\g_{11}\le \dim \ck\le \dim G_{00}{\cdot}y + \dim p^{-1}(y)  \\
   =\dim \g_{00}-\dim\z_\g(y)_{00}+\dim\z_\g(y)_{10} 
    = \dim \g_{11}-\dim\z_\g(y)_{11}+\dim\z_\g(y)_{10} .
 \end{multline*}
The last equality uses Lemma~\ref{prop.5}. % 
Hence, the  existence of such a component $\ck$ would imply that 
%\beq  
    $\dim\z_\g(y)_{11}\le \dim\z_\g(y)_{10}$
%\eeq
for some nonzero $y\in\N_{11}\subset \g_{11}$.
One can rewrite the last condition so that it will only depend on the (inner) involution 
$\sigma_3$. Since $\{\sigma_1,\sigma_2\}$ is a dyad, we have 
$\dim\z_\g(y)_{10}=\dim\z_\g(y)_{01}$; and since $\sigma_3$ is inner and
$\g^{\sigma_3}=\g_{00}\oplus \g_{11}$ is a $\BZ_2$-grading of maximal rank, we have 
$\dim\z_\g(y)_{11}=\dim\z_\g(y)_{00}+\rk\,\g^{\sigma_3}=
\dim\z_\g(y)_{00}+\rk\,\g$, cf. \eqref{ravenstvo-max-rank}. Then 
\[
\dim\z_\g(y)_{11}+\dim\z_\g(y)_{00}+\rk\,\g=2\dim\z_\g(y)_{11}\le 2\dim\z_\g(y)_{10}=
\dim\z_\g(y)_{10}+\dim\z_\g(y)_{01}. 
\]
In other words, if the assumption were true, we would have
\beq  \label{eq:revert}
   \dim (\z_\g(y)\cap \g^{\sigma_3})+\rk\,\g\le \dim (\z_\g(y)\cap \g^{(\sigma_3)}_1) .
\eeq
for some nonzero nilpotent $y\in\g_{11}$. (Note that since $\g^{\sigma_3}=\g_{00}\oplus \g_{11}$ is a $\BZ_2$-grading of maximal rank, $\g_{11}$ meets all nilpotent orbits in $\g^{\sigma_3}$
\cite{leva}. Therefore, {\sl a priori}, $y$ can be {any} nonzero nilpotent element of $\g^{\sigma_3}$.)
However,  Theorem~\ref{thm:strange-ineq} below shows that \eqref{eq:revert}
is never satisfied if $y\ne 0$. This completes the proof of parts (i) and (ii).

For (iii), it suffices to prove that each irreducible component of $\fe$ contains a point 
$(x,y)$ such that $\textsl{d}\vp_{(x,y)}$ is onto, i.e.,
$\Ima\textsl{d}\vp_{(x,y)}=\g_{01}$, cf.~\cite[Lemma\,2.3]{ri81}. Since each irreducible component 
of $\fe$ is the closure of the $G_{00}$-saturation of a homogeneous \CSA\ , it contains a point 
$(x,y)$ such that $\z_\g(x)_{01}\cap\z_\g(y)_{01}=\{0\}$ and then $\textsl{d}\vp_{(x,y)}$ is onto, as 
shown in the proof of Lemma~\ref{lm:dominant}.
\end{proof}

\begin{rmk}
1) For any inner $\sigma\in\mathsf{Inv}(\g)$, there exist commuting involutions of maximal rank $\sigma_1$ and $\sigma_2$ such that $\sigma=\sigma_1\sigma_2$, see Prop.~\ref{prop:povtor}.
Therefore, there are sufficiently many quaternionic decompositions, where 
Theorem~\ref{thm:main-EQ} applies.

2) For an arbitrary dyad $\{\sigma_1,\sigma_2\}$, it can happen that all irreducible components of 
$\fe$ are standard, but they have different dimensions. That is, $\vp:\g_{10}\times\g_{11}\to 
\g_{01}$ is not equidimensional, but still any pair of commuting elements in 
$\g_{10}\times\g_{11}$ can be approximated by a pair of commuting {\sl semisimple\/} elements. 
\end{rmk}

\begin{ex}   \label{ex:son-so(n-1)}
 Let $\sigma_1$ be an involution of $\g=\son$ such that $\g^{\sigma_1}=\mathfrak{so}_{n-1}$.
This can be included in a dyad $\{\sigma_1,\sigma_2\}$ such that $\g^{\sigma_3}=\fso_{n-2}\times
\fso_2$. The quaternionic decomposition is \\
\begin{center}
\setlength{\unitlength}{0.02in}
\begin{picture}(35,24)(0,0)
\put(-20,7){$\g=$}
    \put(-2,15){$\fso_{n-2}$}    \put(30,15){$\mathsf R(\varpi_1)$}
    \put(-3,0){$\mathsf R(\varpi_1)$}      \put(31,0){$\mathsf R(0)$}
\qbezier[30](-2,9),(23,9),(48,9)            % horisontal line
\qbezier[20](23,-3),(23,11),(23,24)      % vertical line
\put(19.9,7){$\oplus$}   %   central plus
\put(21,-9){{\color{my_color}$\sigma_2$}}
\put(53,7){{\color{my_color}$\sigma_1$}\ ,}
\end{picture}
\end{center}
\vskip2.5ex
where the trivial $\fso_{n-2}$-module $\mathsf R(0)$ is just the central torus $\fso_{2}$ in 
$\g^{\sigma_3}$. Here $\dim\ce_{10}=\dim\ce_{11}=1$ and
the zero-fibre of multiplication $\g_{10}\times\g_{11}\to\g_{01}$ consists of two irreducible components,
$\g_{10}\times\{0\}\simeq \bbk^{n-2}$ and $\{0\}\times \g_{11}\simeq\bbk$. Both components are standard.
\end{ex}

The following auxiliary result does not refer to quaternionic decompositions; it concerns the case of a sole
involution. %For this reason, we simplify our notation to the setting of Section~\ref{sect1}. 

\begin{thm}    \label{thm:strange-ineq}
Let $\sigma$ be an arbitrary involution of\/ $\g$ and $\g=\g_0\oplus\g_1$ the corresponding
$\BZ_2$-grading. For any nonzero $x\in\g_0$, we have
\beq    \label{eq:strange-ineq}
   \dim\g_0^x+\rk\,\g-\dim\g_1^x> 0 .
\eeq
More precisely, one always has $\dim\g_0^x+\rk\,\g-\dim\g_1^x \ge 0$ and the equality only 
occurs if $x=0$ and $\sigma$ is of maximal rank.
\end{thm}
\begin{rmk}
For application to Theorem~\ref{thm:main-EQ}, we only need the case when $x$ is nilpotent and
$\sigma$ is inner. But, surprisingly, the assertion appears to be absolutely general. Unfortunately, our proof is not conceptual, after all. Having successfully reduced the problem to non-even nilpotent elements of
$\g_0$, we then resort to case-by-case considerations. Certainly, there must be a better proof!
\end{rmk}
\begin{proof}
Note that $\dim G{\cdot}x$ is even and, therefore, the LHS in \eqref{eq:strange-ineq}
is always even; hence the more accurate assertion is that 
$\dim\g_0^x+\rk\,\g-\dim\g_1^x\ge 2$ for all nonzero $x\in \g_0$.

1$^o$. If $x=0$, then we have $\dim\g_0+\rk\,\g-\dim\g_1\ge 0$, and the equality holds if and only if $\sigma$ is of maximal rank.

2$^o$.  If $x$ is nonzero semisimple, then $\g^x$ is a $\sigma$-stable reductive subalgebra
and $x$ is a central element of $\g^x$ that belongs to $\g_0^x$. Write $\g^x=\z\oplus\es$, where $\es=[\g^x,\g^x]$ and $\z$ is the centre. Then $\dim\z_0 > 0$ and 
\[
  \dim\g_0^x+\rk\,\g-\dim\g_1^x=(\dim\es_0+\rk\,\es-\dim\es_1)+2\dim \z_0 \ge 2 .
\]
\indent 3$^o$. If $x$ is non-nilpotent, then using the Jordan decomposition $x=x_s+x_n$,
we reduce the problem to the same property for the nilpotent element $x_n$ in the
$\sigma$-stable reductive subalgebra $\z_\g(x_s)$.

4$^o$. From now on, we assume that $x=e\in\g_0$ is nonzero and  nilpotent.
Choose an $\tri$-triple $\{e,h,f\}\subset \g_0$.
Suppose that $e$ is even in $\g$, i.e., the eigenvalues of $\ad h$ in $\g$ are even.
Then $\dim \g^h=\dim\g^e$ and $\dim \g_0^h=\dim\g_0^e$. Thus, the assertion is reduced 
to the same assertion for $h\in\g_0$ and we are again in the setting of part 2$^o$.

5$^o$. Suppose that $e$ is even in $\g_0$, but not in $\g$. That is, the eigenvalues of $\ad h$ in $\g_0$ are even, but $\ad h$ has also some odd eigenvalues in $\g_1$.
Decomposing $\g$ into the sum of $\sigma$-stable ideals, we may assume that either $\g$ is simple
or $\g=\es\oplus\es$, where $\es$ is simple and $\sigma$ is the permutation involution. In the second case, if $e$ is even in $\g_0=\Delta(\es)$, then $e$ is also even in $\g$.
Therefore, without loss of generality, we may assume that $\g$ is simple.

Let us decompose $\g_1$ according to the parity of $\ad h$-eigenvalues:
$\g_1=\g_1^{odd}\oplus\g_1^{even}$. By assumption, $\g_1^{odd}\ne 0$.
Then $\tilde \g:=[\g_1^{odd},\g_1^{odd}]\oplus \g_1^{odd}$ is
an ideal of $\g$ that does not meet $\g_1^{even}$. 
Therefore,  $\tilde\g=\g$ and $\g_1^{even}=0$.
Hence  $\g_0^e=(\g^e)^{even}$ and $\g_1^e=(\g^e)^{odd}$.
Consider the $\BZ$-grading of $\g$ determined by the eigenvalues of $h$, 
$\g=\bigoplus_{i\in\BZ}\g(i)$. The $\tri$-theory shows that  
$\dim(\g^e)^{even}=\dim\g(0)$ and $\dim(\g^e)^{odd}=\dim\g(1)$.
Hence $\dim\g_0^e=\dim\g(0)$ and $\dim\g_1^e=\dim\g(1)$. Finally, it follows from Vinberg's lemma
\cite[\S\,2.3]{vi76} that the group $G(0)$ has finitely many orbits in $\g(1)$, whence $\dim\g(1)\le \dim\g(0)$.
Thus, in this case the stronger inequality $\dim\g_0^e\ge \dim\g_1^e$ holds.

6$^o$. Thus, it remains to handle the case in which a nilpotent element $e\in\g_0$ is not even.
Here we do not know an a priori argument and  resort to the case-by-case considerations.

7$^o$.  If $\g$ is a classical  Lie algebra, then the nilpotent orbits in $\g$ and $\g_0$ are parameterised by partitions, and we use the explicit formulae for $\dim\g^e$ and $\dim\g_0^e$ 
in terms of partitions. 
Some of these calculations are presented in Appendix~\ref{app:A}.

8$^o$.  If $\g$ is an exceptional simple Lie algebra, then, for any non-even nilpotent 
element $e\in\g_0$, we determine the corresponding nilpotent orbit in $\g$ and then 
compare the dimensions of $\g_0^e$ and $\dim\g^e$. Being rather boring, the verification is, however, not very difficult.

\un{For $\sigma$ inner}, we use
the seminal work of Dynkin \cite[Tables~16--20]{dy52},
where he computed,  for all simple 3-dimensional
subalgebras in exceptional Lie algebras,
 the ``minimal including regular semisimple subalgebras" and the corresponding
 weighted Dynkin diagrams. See also comments on this article in 
\cite[p.\,309--312]{dynkin-Coll}, where a few errors occurring in \cite{dy52} are corrected.

To convey the idea, consider some examples related to an (inner) 
involution of $\g=\GR{E}{8}$ with $\g_0=\GR{D}{8}=\mathfrak{so}_{16}$. There are 33 
non-even nilpotent orbits in $\g_0$. (Recall that $e\in \mathfrak{so}_{16}$ is non-even if and
only if the partition of $e$ contains both odd and even parts.)

a) Let $e\in \mathfrak{so}_{16}$ be a nilpotent element corresponding to the partition $(11,2,2,1)$.
Using \cite[Cor.\,3.8(a)]{hess76} or \cite[Prop.\,2.4]{KP82}, we obtain $\dim\g_0^e=16$.
This partition also shows that a ``minimal including regular semisimple subalgebra" of 
$\GR{D}{8}$ containing $e$ is of type $\GR{D}{6}+\GR{A}{1}$. (Here $(11,1)$ is the partition 
of the regular nilpotent element of $\GR{D}{6}$ and any pair of equal parts $(n,n)$ gives rise 
to the simple summand $\GR{A}{n-1}$.) %See \cite[Sect.\,3]{aif99} for the general recipe.)  
Then using \cite[Table\,20]{dy52}, we detect the simple 3-dimensional
subalgebra in $\GR{E}{8}$ with ``minimal including regular semisimple subalgebra" of 
type $\GR{D}{6}+\GR{A}{1}$.  The corresponding nilpotent orbit has the modern label $\GR{E}{7}(a_3)$ and here 
$\dim\g^e=28$. Hence $\dim\g_1^e=12$ and Eq.~\eqref{eq:strange-ineq} holds.

b) Let $e\in \mathfrak{so}_{16}$  correspond to the partition $(7,5,2,2)$.
By \cite[Cor.\,3.8(a)]{hess76},  $\dim\g_0^e=22$. Here a
``minimal including regular semisimple subalgebra" is of type $\GR{D}{6}(a_2)+\GR{A}{1}$,
because the partition $(7,5)$ determines the distinguished nilpotent orbit in $\GR{D}{6}$, 
which is denoted
by $\GR{D}{6}(a_2)$. Using  \cite[Table\,20]{dy52}, we detect the corresponding 
nilpotent orbit in $\g$. This orbit is denoted nowadays by $\GR{E}{7}(a_5)$ and here 
$\dim\g^e=42$.

c) Let $e\in \mathfrak{so}_{16}$  correspond to the partition $(7,4,4,1)$.
By \cite[Cor.\,3.8(a)]{hess76},  $\dim\g_0^e=22$. Here a
``minimal including regular semisimple subalgebra" is of type $\GR{D}{4}+\GR{A}{3}$. 
Using  \cite[Table\,20]{dy52}, we detect the corresponding 
nilpotent orbit in $\g$. This orbit is denoted nowadays by $\GR{D}{6}(a_2)$ and here 
$\dim\g^e=44$. 

\un{If $\sigma$ is outer}, then $\g$ is of type $\GR{E}{6}$. In the respective two cases, we use
the information on $e\in\g_0$ for decomposing $\g_1$ as $\langle e,h,f\rangle$-module,
which allows to compute $\dim\g_1^e$.
\end{proof}

A case-free proof of Theorem~\ref{thm:strange-ineq} might be derived from the following
conjectural invariant-theoretic property of centralisers. Recall that $\g=\g_0\oplus\g_1$ and 
$e\in\g_0$.
Let $G^e_0$ be the connected subgroup of $G_0$ with Lie algebra $\g_0^e$. Then $G^e_0$ acts
on  $(\g_1^e)^*$ and we write $\bbk((\g_1^e)^*)^{G_0^e}$ for the field of $G_0^e$-invariant 
rational functions on $(\g_1^e)^*$.

\begin{conj}     \label{conj:a-la-alela}
For any $e\in\g_0\cap\N$, we have \ $ \trdeg \bbk((\g_1^e)^*)^{G_0^e}\le \rk\,\g$.
\end{conj}

\noindent
By Rosenlicht's theorem \cite[Ch.\,I.6]{brion}, 
$\displaystyle
\trdeg \bbk((\g_1^e)^*)^{G_0^e}=\dim \g_1^e -\max_{\xi\in (\g_1^e)^*} \dim G_0^e{\cdot}\xi$. 
If $e\ne 0$, then the one-dimensional unipotent
group $\exp(te)$, $t\in\bbk$, acts  trivially on $\g_1^e$ and hence 
$\max_{\xi\in (\g_1^e)^*} \dim G_0^e{\cdot}\xi \le \dim\g_0^e -1$. Therefore, if the conjecture
were true, we would obtain 
$\dim\g_1^e-\dim\g_0^e+1\le \rk\,\g$, as required.
Perhaps, this can be related to the Elashvili conjecture, which asserts that 
$\trdeg\bbk((\g^e)^*)^{G^e}=\rk\,\g$ for all $e\in \N$. 

\begin{rmk}
Inequality \eqref{eq:strange-ineq} can be written as $\dim\g^x_0 > \dim\mathcal B_x$, where
$\mathcal B_x$ is the variety of  Borel subalgebras of $\g$ containing $x$ (the Springer fibre of $x$).
[Recall that $\dim\mathcal B_x=(\dim\g^x-\rk\,\g)/2$.]
\end{rmk}

%%%%%%%%%%%%%%%%   Section 4
\section{Commuting varieties and  restricted root systems}  
\label{sect4}

\noindent
Here we assume that $\{\sigma_1,\sigma_2\}$ is a dyad. 
As above, we consider the commutator map $\vp:\g_{10}\times\g_{11}\to \g_{01}$ and
the $\vecs$-commuting variety $\fe=\vp^{-1}(0)$. Then $\dim\fe\ge \dim \g_{11}$ and $\fe$  
have a standard irreducible component 
of expected dimension $\dim\g_{11}$; namely, $\{0\}\times \g_{11}$, see Proposition~\ref{lm:dyad-onto}. 

In this section, we describe a method for detecting subvarieties of $\fe$ of
large dimension. This method is based on comparing restricted root systems for little and big symmetric spaces related to the quaternionic decomposition in question.

Take a little \CSS\ $\ce_{11}\subset\g_{11} $. Then, by Theorem~\ref{thm:povtor}(2), 
$\ce_{11}$ is also a \CSS\ in $\g_{1\star}$ and $\g_{\star 1}$, which is equivalent to that 
$\z_\g(\ce_{11})_{10}=\z_\g(\ce_{11})_{01}=\{0\}$ and $\z_\g(\ce_{11})_{11}=\ce_{11}$.
Our idea is to replace $\ce_{11}$ with a proper subspace $\tilde\ce$ such that:
\vskip1.5ex
\hbox to \textwidth{ \refstepcounter{equation}(\theequation) \hfil 
\label{eq:cond-larger}
%\parbox{12cm}{\textbullet \quad 
$\tilde\ce$ still contains $G_{00}$-regular elements.
\hfil}
\vskip1.5ex
\noindent
Then we consider $\hat\ce:=\z_\g(\tilde\ce)_{10}\times\tilde\ce \subset \fe$ 
and compute the dimension of 
%a straighforward calculation shows that
$G_{00}{\cdot} \hat\ce$. %Indeed, 
Since $\ov{G_{00}{\cdot}\ce_{11}}=\g_{11}$, we have 
\[
    \dim G_{00}+\dim\ce_{11}-\dim\z_\g(\ce_{11})_{00}=\dim\g_{11} .
\]
Set $\mathfrak T_{00}(\hat \ce)=\{g\in G_{00} \mid g{\cdot}y\in \hat\ce \ \text{ for generic } \ 
y\in\hat\ce\}$, and likewise for $\ce_{11}$. In view of 
\eqref{eq:cond-larger}, we have 
$\dim\mathfrak T_{00}(\hat \ce)=\dim\mathfrak T_{00}(\ce_{11})=\dim\z_\g(\ce_{11})_{00}$.
Then
\begin{multline}   \label{eq:larger-comp-raznesenie}
  \dim G_{00}{\cdot} \hat\ce=\dim G_{00}+\dim\hat\ce%\z_\g(\tilde\ce)_{10}+\dim\tilde\ce - 
  -\dim \mathfrak T_{00}(\hat \ce) \\  
  =\bigl(\dim G_{00}+\dim\ce_{11}-\dim\z_\g(\ce_{11})_{00}\bigr) + \bigl(\dim\z_\g(\tilde\ce)_{10}-
  \dim\ce_{11}+\dim\tilde\ce\bigr)\\  
  =\dim\g_{11} + \bigl(\dim\hat\ce - \dim\ce_{11}\bigr)%> \dim\g_{11} .
\end{multline}

Thus, we obtain a subvariety of larger dimension, if $\dim\z_\g(\tilde\ce)_{10} +\dim\tilde\ce > \dim\ce_{11}$.
Of course, it is not always possible to construct such a $\tilde\ce$. Our sufficient condition 
exploits restricted  root systems. Set $\h=\g^{\sigma_3}$, and let $H$ denote the corresponding 
connected (reductive) subgroup of $G$.
Write $\bar\sigma$ for the restriction to $H$ of $\sigma_1$ or $\sigma_2$.

Let $C_{11}=\exp(\ce_{11})\subset H\subset G$ be the corresponding torus. The coincidence 
of \CSS\ means that $C_{11}$ is a maximal $\sigma_1$-anisotropic torus in $G$ and a maximal 
$\bar\sigma$-anisotropic torus in $H$.
Accordingly, one obtains the inclusion of two restricted root systems relative to $C_{11}$:
\[
  \Psi (H/G_{00})\subset \Psi(G/G_{0\star}) .
\]
Identifying restricted roots and their differentials, one may consider restricted roots as linear 
forms on $\ce_{11}$. Then the set of $G_{00}$-regular elements of $\ce_{11}$ is 
$\{x\in \ce_{11} \mid \mu(x){\ne}0 \ \ \forall \mu\in \Psi(H/G_{00})\}$ and the set of
$G_{0\star}$-regular elements %in $\ce_{11}$ 
is  $\{x\in \ce_{11} \mid \mu(x){\ne} 0 \ \ \forall \mu\in \Psi(G/G_{0\star})\}$.

\begin{prop}                     \label{prop:larger-comp}
Assume that $\mu\in \Psi(G/G_{0\star})$ and $r\mu\not\in \Psi(H/G_{00})$
for any $r\in\mathbb Q$. If $m_\mu >1$, then $\dim\fe\ge \dim\g_{11}+m_\mu-1 > \dim\g_{11}$.
\end{prop}
\begin{proof}
Under this assumption,
$\tilde\ce:=\Ker(\mu) \subset \ce_{11}$ still contains 
$G_{00}$-regular elements, and $\dim\tilde\ce=\dim\ce_{11}-1$. Furthermore, $\z_\g(\tilde\ce)$ is $\vecs$-stable and
$\z_\g(\tilde\ce)=\z_\g(\ce_{11})\oplus \g_\mu\oplus\g_{-\mu}$. Recall that
$\z_\g(\ce_{11})$ is contained in $\g_{00}\oplus\g_{11}$.
Clearly, $\g_\mu\oplus\g_{-\mu}$ is also  $\vecs$-stable and is contained in 
$\g_{01}\oplus \g_{10}$. 

Since $\{\sigma_1,\sigma_2\}$ is a dyad, 
$\dim (\g_\mu\oplus\g_{-\mu})\cap \g_{10}=\dim (\g_\mu\oplus\g_{-\mu})\cap \g_{01}=m_\mu$.
Hence $\dim\z_\g(\tilde\ce)_{10}=m_\mu$, and the assertion follows from 
Eq.~\eqref{eq:larger-comp-raznesenie}.
\end{proof}

\begin{rmk}    \label{rmk:obv-modific}
1) Such a construction gives nothing, if all root multiplicities in $\Psi(G/G_{0\star})$ are equal to $1$. For instance, if $\sigma_1$ is of maximal rank.

2) The procedure described in the previous proof admits obvious modifications. Roughly 
speaking, if there are linearly independent roots $\mu_1,\mu_2,\dots$ in $\Psi(G/G_{0\star})$, 
with large multiplicities, such that 
$\mathbb Q\textsf{-span}\{\mu_1,\mu_2,\dots\}\cap \Psi(H/G_{00})=\varnothing$,  
then one can take $\tilde\ce=\Ker(\mu_1,\mu_2,\dots)$, see Proposition~\ref{prop:larger-comp2} below.
\end{rmk}

Although it is convenient to stick to one specific $\vecs$-commuting variety in theoretical considerations, it may happen that in concrete examples different $\vecs$-commuting varieties
exhibit different good (or bad) properties.

\begin{ex}    \label{ex:sl-sp}
Let $\sigma_1$ be an outer involution of $\g=\sltn$ with $\g^{\sigma_1}=\spn$. 
In~\cite[Sect.\,2]{comm-inv1}, we gave a method for describing all the dyads including $\sigma_1$,
which exploits the restricted root system $\Psi(G/G^{\sigma_1})$. This implies that one can find 
$\sigma_2$ conjugated $\sigma_1$ such that 
the inner involution $\sigma_3=\sigma_1\sigma_2$ has the fixed-point subalgebra
$\h=\mathfrak{sl}_{2m}\oplus \mathfrak{sl}_{2n-2m}\oplus\te_1$. The corresponding quaternionic decomposition is
\\[.7ex]
\begin{center}
\begin{picture}(140,25)(0,2) \setlength{\unitlength}{0.02in}
\put(-37,7){$\sltn=$}
    \put(-5,15){$\mathfrak{sp}_{2m}\oplus\mathfrak{sp}_{2n-2m}$}    
    \put(65,15){$\mathsf R(\varpi_1)\mathsf R(\varpi_1')$}
    \put(2,-2){$\mathsf R(\varpi_1)\mathsf R(\varpi_1')$}      
    \put(60,-2){$\mathsf R(\varpi_2){+}\mathsf R(\varpi_2'){+}\mathsf R(0)$}
\qbezier[60](-13,9),(53,9),(122,9)            % horisontal line
\qbezier[20](53,-8),(53,12),(53,24)      % vertical line
\put(49.9,7){$\oplus$}   %   central plus
\put(51,-14){{\color{my_color}$\sigma_2$}}
\put(132,7){{\color{my_color}$\sigma_1$}}
\end{picture}
\end{center}
\vskip3.5ex
where $\varpi_i$ (resp. $\varpi_i'$) are fundamental weights of $\mathfrak{sp}_{2m}$
(resp. $\mathfrak{sp}_{2n-2m}$), and $\mathsf R(\lb)$ is a simple module of the
respective simple Lie algebra with highest weight $\lb$. 

\textbullet \ \ Here $G=SL_{2n}$, $G_{0\star}=Sp_{2n}$, $H=SL_{2m}\times SL_{2(n-m)}\times T_1$, and
$G_{00}=Sp_{2m}\times Sp_{2(n-m)}$.
According to Table~VI in \cite[Ch.\,X]{helg},
we have $\Psi(G/G_{0\star})=\GR{A}{n-1}$, $\Psi(H/G_{00})=\GR{A}{m-1}+\GR{A}{n-m-1}$, and
all root multiplicities in $\Psi(G/G_{0\star})$ equals $4$.
Since $\Psi(H/G_{00})$ has fewer roots, Proposition~\ref{prop:larger-comp} implies that 
$\fe$ has an irreducible component of dimension $>\dim\g_{11}+(4-1)$ and our test 
map $\vp:\g_{10}\times\g_{11}\to \g_{01}$ is not equidimensional.

\textbullet \ \ 
Here $\dim\ce_{01}=\dim\ce_{10}=\min\{m,n-m\}$ and  any big \CSS\ in $\g_{10}\oplus\g_{01}$ is 
of dimension $2\min\{m,n-m\}$. By Theorem~\ref{thm:krit-odno-homog-CSS}(2), this means that 
all homogeneous \CSS\ in $\g_{10}\oplus\g_{01}$ are  $G_{00}$-conjugate, and therefore,
the $\vecs$-commuting variety related to the  commutator $\g_{10}\oplus\g_{01}\to \g_{11}$
has a unique standard component.
%$\dim\ce_{\ast,1-\ast}=2m$.
\end{ex}

\begin{ex}    \label{ex:E7-D6A1}
Let $\sigma$ be an involution of $\g=\GR{E}{7}$ with $\g^\sigma=\GR{D}{6}\times\GR{A}{1}$.
It can be included in two non-conjugate triads \cite{kollross}. One of them has
$\g_{00}=\GR{D}{4}\times\GR{A}{1}^3$,  with quaternionic decomposition
\\[.6ex]
\begin{center}
\begin{picture}(140,26)(0,2) \setlength{\unitlength}{0.02in}
\put(-35,7){$\GR{E}{7}=$}
    \put(15,15){$\GR{D}{4}\times\GR{A}{1}^3$}    
    \put(60,15){$\mathsf R(\varpi_4)\mathsf R(\varpi)\mathsf R(\varpi'')$}
    \put(-12,-2){$\mathsf R(\varpi_3)\mathsf R(\varpi)\mathsf R(\varpi')$}      
    \put(60,-2){$\mathsf R(\varpi_1)\mathsf R(\varpi')\mathsf R(\varpi'')$}
\qbezier[60](-13,9),(53,9),(122,9)            % horisontal line
\qbezier[20](53,-8),(53,12),(53,24)      % vertical line
\put(49.9,7){$\oplus$}   %   central plus
\put(51,-14){{\color{my_color}$\sigma_2$}}
\put(130,7){{\color{my_color}$\sigma_1$}}
\end{picture}
\end{center}
\vskip3.5ex
where $\varpi,\varpi',\varpi''$ are the fundamental weights of the simple factors of $\GR{A}{1}^3$,
and $\varpi_i$'s are fundamental weights of $\GR{D}{4}$. 
Here $\dim\g_{ij}=32$ for $(ij)\ne(00)$ and our test commutator map is
\[
    \vp: \mathsf R(\varpi_3)\mathsf R(\varpi)\mathsf R(\varpi') \times
    \mathsf R(\varpi_1)\mathsf R(\varpi')\mathsf R(\varpi'') \to
    \mathsf R(\varpi_4)\mathsf R(\varpi)\mathsf R(\varpi'') .
\]
Using Table~VI in \cite[Ch.\,X]{helg}, we find that 
$\rk(\GR{E}{7}/\GR{D}{6}\times\GR{A}{1})=4$ and
the restricted root system $\Psi (\GR{E}{7}/\GR{D}{6}\times\GR{A}{1})$ is of type $\GR{F}{4}$;
whereas $\rk(\GR{D}{6}\times\GR{A}{1}/\GR{D}{4}\times\GR{A}{1}^3)=
\rk(\GR{D}{6}/\GR{D}{4}\times\GR{A}{1}^2)=4$ and
the corresponding root system is of type $\GR{B}{4}$. 
The long (resp. short) roots of $\GR{B}{4}$ are also long (resp. short) roots of $\GR{F}{4}$, and
the multiplicities are $m_{\text {long}}=1$, $m_{\text {short}}=4$.
However, the root system $\GR{B}{4}$ has fewer short roots than $\GR{F}{4}$. Therefore,
Proposition~\ref{prop:larger-comp} applies here, and $\fe$ has an irreducible component of dimension at least $\m_{\text {short}}-1+\dim\g_{11}=35$. 
\end{ex}
\begin{ex}   \label{ex:F4-B4}
Let $\sigma$ be an involution of $\g=\GR{F}{4}$ with $\g^\sigma=\GR{B}{4}=\mathfrak{so}_9$.
Up to conjugacy, this involution can be included in a unique triad \cite{kollross}, with quaternionic decomposition
\begin{center}
\begin{picture}(140,28)(0,5) \setlength{\unitlength}{0.02in}
\put(2,7){$\GR{F}{4}=$}
    \put(30,15){$\GR{D}{4}$}    \put(60,15){$\mathsf R(\varpi_4)$}
    \put(28,0){$\mathsf R(\varpi_3)$}      \put(60,0){$\mathsf R(\varpi_1)$}
\qbezier[34](20,9),(50,9),(85,9)            % horisontal line
\qbezier[20](53,-3),(53,11),(53,24)      % vertical line
\put(49.9,7){$\oplus$}   %   central plus
\put(51,-10){{\color{my_color}$\sigma_2$}}
\put(90,7){{\color{my_color}$\sigma_1$}}
\end{picture}\ \ ,
\end{center}
\vskip3ex
where $\dim\mathsf R(\varpi_i)=8$ and the main diagonal represent the little involution of 
$\g^{\sigma_3}=\GR{B}{4}=\mathfrak{so}_9$.
Our test commutator is the bilinear $\GR{D}{4}$-equivariant mapping
$\mathsf R(\varpi_3)\times \mathsf R(\varpi_1)\to \mathsf R(\varpi_4)$.
Here $\rk(\GR{F}{4}/\GR{B}{4})=1$ and the restricted root system 
$\Psi (\GR{F}{4}/\GR{B}{4})$ is of type $\GR{BC}{1}$.
The restricted root system $\Psi (\GR{B}{4}/\GR{D}{4})$ is of type $\GR{C}{1}$. Since all little and 
big \CSS\ are one-dimensional, Proposition~\ref{prop:larger-comp} does not help here. Actually, 
the only standard components of $\fe$ are $\g_{10}\times\{0\}$ and 
$\{0\}\times \g_{11}$, both of dimension $8$. Below, we describe an ``intermediate'' non-standard 
irreducible component of dimension $11$. 

Let $x\in\g_{11}\simeq \mathsf R(\varpi_1)$ be a nonzero nilpotent element.
All such elements form a sole $7$-dimensional $SO_8$-orbit. By Lemma~\ref{prop.5},
$\dim SO_9{\cdot}x=2{\cdot}7=14$ and hence $\dim(\mathfrak{so}_9)^x=22$. The only nilpotent $SO_9$-orbit of dimension $14$ in 
$\mathfrak{so}_9$ is the orbit of short root vectors. The short roots of $\g^{\sigma_3}=\GR{B}{4}$ are also short
roots of $\g=\GR{F}{4}$. Therefore, a ``minimal including regular semisimple subalgebra'' is 
$\GRt{A}{1}$ in Dynkin's notation. This implies that
$\dim \z_\g(x)=30$ and completely determines the dimension matrix of  the spaces $\z_\g(x)_{ij}$: \ 
$\begin{array}{c|c}
21 & 4 \\  \hline %\hdotsfor{3} \\
4 & 1 \end{array}$.  Here the 1-dimensional space $\g_{11}$ is just the line $\bbk x$.
Then $\dim \ov{G_{00}{\cdot}(\z_\g(x)_{10}\oplus \bbk x)}=4+7=11$. Using the projection 
$\fe\to \g_{11}$, one can prove that 
$\ov{G_{00}{\cdot}(\z_\g(x)_{10}\oplus \bbk x)}$ is the only new irreducible component of $\fe$.
It is contained in $\N_{10}\times\N_{11}$. Thus, $\fe$ has three irreducible components.
\end{ex}

 %%%%%%%%%%%%%%%%   Section 5
\section{Triads of Hermitian involutions and simple Jordan algebras}  
\label{sect5}

\noindent
In this section, $\g$ is assumed to be simple.
We say that $\sigma\in \mathsf{Inv}(\g)$ is {\it Hermitian\/} if $\g_0$ is not semisimple.
All these involutions are associated with $\BZ$-gradings of $\g$ with only three nonzero
terms ({\it short gradings}), i.e., with parabolic subalgebras with abelian nilpotent radical.
Let $\g=\g(-1)\oplus\g(0)\oplus\g(1)$ be a short grading. Then $\p=\g(0)\oplus\g(1)$ is a
(maximal) parabolic subalgebra with abelian nilpotent radical, and one defines a Hermitian 
involution
$\sigma$ by letting $\g^{\sigma}=\g(0)$, $\g_1^{(\sigma)}=\g(-1)\oplus\g(1)$.

Since $\g$ is simple, the centre of $\g(0)$ is one-dimensional and there is a {\sl unique\/} 
$h\in\g(0)$ such that $\g(i)=\{x\in \g\mid [h,x]=2ix\}$. 
By \cite[\S\,2.3]{vi76}, the reductive group $G(0)$ has finitely many orbits in $\g(1)$.
Let $\co$ be the dense $G(0)$-orbit in $\g(1)$ and $e\in\co$. Set $\g(i)^e=\g(i)\cap \g^e$.

For future reference, we provide a proof of the following well-known assertion.

\begin{lm}
 $h\in [\g,e] \ \Longleftrightarrow \ \g(0)^e$ \ is reductive. 
\end{lm}
\begin{proof}
1)  If $h\in [\g,e]$, then $h=[e,f]$ for some $f\in\g(-1)$ and therefore, $\{e,h,f\}$ is an $\tri$-triple.
Then $\g(0)^e=\z_\g(e,h,f)$, which is reductive.

2) For $e\in\co$, we have $\dim\g(0)^e=\dim\g(0)-\dim\g(1)$. Using the Kirillov--Kostant
{form} associated with $e$, we see that $\dim\g(-1)-\dim\g(-1)^e=\dim\g(0)-\dim\g(0)^e$. Hence
$\g(-1)^e=0$ and $\g^e=\g(0)^e\oplus \g(1)$. Set $\ka=\g(0)^e$, and  let $(\ )^\perp$
denote the orthocomplement with respect the Killing form.
Then $[\g,e]=(\g^e)^\perp=\g(1)\oplus (\ka^\perp\cap\g(0))$. 
Now, if $\ka$ is reductive, then the restriction of the Killing form to $\ka$ is non-degenerate and $\m:=\ka^\perp\cap\g(0)$ is a $\ka$-stable complement to $\ka$ in $\g(0)$.
Since $\dim [\g(-1),e]=\dim \g(1)=\dim\g(0)-\dim\ka$, we conclude that $\m=[\g(-1),e]$.  
Thus, $e$ acts on $\g$ as follows:
\beq  \label{eq:strukt-e-mod}
    \left\{\begin{array}{cc} 
    \g(-1) \stackrel{\sim}{\longrightarrow} \m \stackrel{\sim}{\longrightarrow} \g(1)\to 0\\
                                     \ka\to 0
    \end{array} \right.
\eeq
Let $\{e,\tilde h,f\}$ be an $\tri$-triple with $\tilde h\in \g(0)$ and $f\in \g(-1)$.
Such a triple always exists, see \cite[\S\,2]{vi79}. Then Eq.~\eqref{eq:strukt-e-mod} shows 
that $\g$ is a sum of $3$-dimensional and $1$-dimensional $\tri$-modules, and that
$\g^{\tilde h}=\ka\oplus\m$. Since $\g(0)$ has one-dimensional centre, one must have 
$\tilde h=h$. Thus, $h\in [\g,e]$. 
\end{proof}

\begin{thm}   \label{thm:herm-triad}
Suppose that a Hermitian involution $\sigma=\sigma_1$ has the property that\/ $\g(0)^e$ \ is reductive. Then $\sigma_1$ can be included in a triad.
\end{thm}
\begin{proof}
Using the notation of the previous proof, we set $\ka=\g(0)^e$ and take (the unique)
$f\in\g(-1)$ such that $h=[e,f]$. Then $\{e,h,f\}$ is an $\tri$-triple,
%Set $\ka=\g(0)^e=\z_\g(e,h)=\z_\g(e,h,f)$. 
$[e,\g(-1)]=:\m$ is a complementary $\ka$-submodule to $\ka$ in $\g(0)$, and
$[e,[e,\g(-1)]]=\g(1)$. This also shows that $\g(-1)$, $\m$, and $\g(1)$ are isomorphic $\ka$-modules. 

In this case, $\ka$ is the fixed-point subalgebra of an involution of $\g(0)$ and
the $(-1)$-eigenspace of this involution is $\m$  
(see Proof of Prop.\,3.3 in~\cite{manuscr94}).
Let $\sigma_2$ denote this involution of $\g(0)$. Then $\sigma_2(h)=-h$.
We extend $\sigma_2$ to the whole of $\g$ by letting $\sigma_2(e)=f$.
Then $\sigma_2([x,e])=[-x,f]$ for all $x\in\m$, which defines $\sigma_2$ on  $\g(1)$
and shows that $\sigma_2(\g(1))\subset \g(-1)$.
Clearly, $\sigma_1$ and $\sigma_2$ commute. Furthermore, $\sigma_1$ and $\sigma_2$
are different involutions of the $3$-dimensional simple subalgebra $\langle e,h,f\rangle$. This implies that 
$\sigma_1$, $\sigma_2$, and $\sigma_3=\sigma_1\sigma_2$ are already conjugate with respect
to $PSL_2=\mathsf{Aut}\langle e,h,f\rangle$. In particular, $\{\sigma_1,\sigma_2,\sigma_3\}$ is a triad.
\end{proof}

This theorem can be derived from the classification of triads, but our direct construction 
allows to visualise the resulting quaternionic decomposition rather explicitly. We have

\hbox to \textwidth{\enspace \refstepcounter{equation} (\theequation) 
\hfil   \label{eq:hermitian-triple}
%\begin{figure}[h]  
{\setlength{\unitlength}{0.02in}
%\begin{center}
\begin{picture}(55,30)(0,5)
\put(-32,10){$\g=$}
    \put(3,17){$\ka$}    \put(36,17){$\m$}
    \put(-15,2){$[\m, e-f]$}      \put(28,2){$[\m,e+f]$}
\qbezier[35](-15,11),(22,11),(60,11)            % horisontal line
\qbezier[20](23,1),(23,15),(23,29)      % vertical line
\put(19.9,9){$\oplus$}   %   central plus
\put(21,-8){{\color{my_color}$\sigma_2$}}
\put(65,9){{\color{my_color}$\sigma_1$}}
\end{picture}  \hfil
}}
\vskip3ex
\noindent
Here $h\in\m=\g_{01}$, $e+f\in [\m, e-f]=\g_{10}$, and $e-f\in [\m,e+f]=\g_{11}$. Note also that
$\ka\oplus\m=\g(0)$ and $ [\m, e-f]\oplus [\m,e+f]=\g(1)\oplus \g(-1)$.

{\bf Remark.} If $\g(0)^e$ is not reductive, then such a triad may not exist. For instance, if
$\g=\mathfrak{sl}_{2n}$ and $\g_0=\mathfrak{sl}_{m}{\times}\mathfrak{sl}_{2n-m}{\times}\te_1$ with
$n\ne m$ and $m$ odd, then there is no respective triad, see \cite[3.2]{vinb}.

As is well known, if $\g(0)^e$ is reductive, then $\g(-1)$ has a structure of a simple 
Jordan algebra. Namely, for $x,y\in \g(-1)$, we set 
\[
         x\circ y=[x,[e,y]] \in \g(-1) .
\]
Then $\{\g(-1), \circ\}$ is a simple Jordan algebra \cite{tits,kantor}.
(See also \cite[Sect.\,4]{kac80} for possible generalisations). Here 
$\ka=\g_{00}$ is the Lie algebra of derivations of $\{\g(-1),\circ\}$.
The triad constructed in Theorem~\ref{thm:herm-triad} is called a {\it Jordan triad}.

\begin{df}  \label{def:Jordan}
The {\it commuting variety\/} of a Jordan algebra $\{\eus J, \circ\}$ is 
\\[.6ex]
\centerline{$\fe(\eus J)=\{(x,y)\mid x\circ y=0\}\subset \eus J\times \eus J$.}
\end{df}

The Jordan triad \eqref{eq:hermitian-triple} provides a  link between the commutator
mapping $\vp:\g_{10}\times \g_{11} \to \g_{01}$ and the commuting variety of the simple Jordan 
algebra $\g(-1)$.

\begin{thm}  \label{thm:jord-com-var}
The commuting variety of the Jordan algebra $\{\g(-1),\circ\}$
is isomorphic to the zero fibre of the commutator
mapping $\vp: \g_{10}\times \g_{11}=[\m,e-f]\times [\m, e+f] \to \m=\g_{01}$.
\end{thm}
\begin{proof}
Any element of $\m$ can uniquely be written as $[x,e]$ with $x\in \g(-1)$. So, if 
$[x,e], [y,e]\in \m$ are arbitrary, then $[[x,e],e-f]\in \g_{10}$ and $[[y,e],e+f]\in \g_{11}$  are 
arbitrary and $\vp$ takes the corresponding pair to 
$\bigl[[[x,e],e-f],[[y,e],e+f]\bigr]\in\m=\g_{01}$. It is a good exercise in the Jacobi identity to check
that 
\[
    \bigl[[[x,e],e-f],[[y,e],e+f]\bigr]=2\bigl[ [[x,e],y],e\bigr] .
\]
(One should use the fact that $h=[e,f]$ is the defining element of the short grading. Hence
$[[x,e],f]=2x$, etc.)
Since $a=[[x,e],y]\in \g(-1)$ and $\g^e\cap \g(-1)=0$, we have $[a,e]=0$ if and only if $a=0$.
Therefore, 

$([[x,e],e-f],[[y,e],e+f])\in \vp^{-1}(0)\Leftrightarrow  [[x,e],y] =0\Leftrightarrow (x,y)\in \fe(\g(-1))$.
\end{proof}

If $\eus J$ is a simple Jordan algebra, then the operator $L_x:\eus J\to\eus J$, $L_x(y)=x\circ y$,
is invertible for almost all $x$.
%there is $x\in \eus J$ such that the condition $x\circ y=0$ implies $y=0$. 
Therefore, $\eus J\times\{0\}$ and $\{0\}\times\eus J$ 
are two irreducible components of $\fe (\eus J)$. Clearly, there are some other irreducible 
components. It is an interesting problem to determine all the components 
of $\fe (\eus J)$ and their dimensions.

The  list of Hermitian involutions leading to Jordan triads and simple Jordan algebras is given 
below. We point out the semisimple subalgebra $\es=[\g(0),\g(0)]$ and the structure of $\g(1)$ as 
$\es$-module. Here 
the $\varpi_i$'s are the fundamental weights of $\es$.
\begin{center}
\begin{tabular}{c|ccccc|c}
& $\g$ & $\es$ & $\g(1)$ & $\ka$ & $\eus J$ & \\ \hline
1 & $\sltn$ & $\sln{\oplus}\sln$ & $\mathsf R(\varpi_1){\otimes}\mathsf R(\varpi_1')$ & $\sln$ & 
$n{\times} n $-matrices  & $n\ge 1$ \\
2 & $\spn$ & $\sln$ & $\mathsf R(2\varpi_1)$ & $\son$ & symmetric $n{\times} n $-matrices & $n\ge 2$\\
3 & $\mathfrak{so}_{4n}$ & $\sltn$ & $\mathsf R(\varpi_2)$ & $\spn$  
&  skew-symm. $2n{\times} 2n $-matrices & $n\ge 2$ \\
4 & $\mathfrak{so}_{n+2}$ & $\son$ &  $\mathsf R(\varpi_1)$ & $\mathfrak{so}_{n-1}$ & spin-factor & $n\ge 4$ \\
5 & $\GR{E}{7}$ & $\GR{E}{6}$ & $\mathsf R(\varpi_1)$ & $\GR{F}{4}$ & the Albert algebra & \\ 
\hline
\end{tabular}
\end{center}
\vskip1ex
\begin{rema}
The Jordan multiplication in the space $\mathsf{Skew}_{2n}$ of usual skew-symmetric matrices is 
defined as follows. If $A,B,J\in \mathsf{Skew}_{2n}$ and $J$ is non-degenerate, then
$A\circ B=\frac{1}{2}(AJB+BJA)$.
\end{rema}
There are some coincidences for small $n$. Namely, 
\begin{center}
Item~1 ($n=1$) \ $\simeq$ \ Item~2 ($n=1$) , \
Item~1 ($n=2$) \ $\simeq$ \ Item~4 ($n=3$) .
\end{center}
Furthermore, if $n=1$ in Item~3, then $\g$ is not simple. This explains the conditions on $n$ given in the last column.
For Item~2, the Hermitian involution (of $\spn$) is of maximal rank and
the respective Jordan algebra %$\eus J$ 
is the algebra $\mathsf{Sym}_n$ of symmetric $n\times n$-matrices.
Therefore, by Theorems~\ref{thm:main-EQ} and \ref{thm:jord-com-var}, the 
multiplication morphism
$\circ:  \mathsf{Sym}_n\times \mathsf{Sym}_n\to \mathsf{Sym}_n$ is equidimensional, i.e., 
$\dim\fe(\mathsf{Sym}_n)=\dim\mathsf{Sym}_n=(n^2+n)/2$.

In all other cases, the multiplication morphism $\eus J\times\eus J\to \eus J$ is not 
equidimensional, see Proposition~\ref{prop:larger-comp2}. Before checking this,
we give an "elementary" explanation for the Jordan algebra of all matrices (Item 1).

\begin{ex}   \label{ex:jordan-algebra-matrizen} 
Let $\mathsf M$ be the associative (also Lie and Jordan) algebra  of all $n\times n$-matrices.
That is, we exploit the usual matrix product, the Lie bracket $[A,B]=AB-BA$, and the Jordan product $A\circ B=(AB+BA)/2$.
Let  $\chi(B)=\det (\lb I- B)=\sum_i \chi_{n-i}(B)\lb^i$ be the characteristic polynomial of a matrix $B$.
Let $\z^{J}(B)$ and $\z^{Lie}(B)$ denote the Jordan and Lie centraliser of $B$, respectively.
Consider the subvariety 
\[
   \mathsf M^{\langle 2\rangle}=\{B\in \mathsf M \mid \chi_{2i+1}(B)=0 \ \forall i\}.
\]
It is an irreducible complete intersection and $\codim \mathsf M^{\langle 2\rangle}=[n+1/2]$
(cf. \cite[Lemma\,5.3]{ri87}).
We also need the dense open subset $\mathsf M^{reg}$ of regular elements (in the Lie algebra sense)
and the subvariety 
\[
    \mathsf M^{ev}=\{ B\in  \mathsf M \mid B \ \text{ is conjugate to } \ {-}B\}.
\]
If $B\in \mathsf M^{ev}$ and $ABA^{-1}=-B$, then $A\in \z^J(B)$ and the mapping
$C\in \z^{Lie}(B) \mapsto AC \in \z^{J}(B)$ is a linear isomorphism. In particular,
$\dim \z^{J}(B)=\dim \z^{Lie}(B)$.
The following is clear:
\begin{itemize}
\item \ $\mathsf M^{\langle 2\rangle}\cap \mathsf M^{reg} \ne\varnothing$ (it contains a regular nilpotent element);
\item \ $\mathsf M^{ev}\subset \mathsf M^{\langle 2\rangle}$ and 
$\mathsf M^{ev}\cap \mathsf M^{reg} \ne\varnothing$;
\end{itemize}
\begin{utve} We have 
 $\mathsf M^{\langle 2\rangle}\cap \mathsf M^{reg}\subset \mathsf M^{ev}$. In particular, $\dim\z^J(B)=n$ for almost all $B\in \mathsf M^{\langle 2\rangle}$.
\end{utve}
\begin{proof} If $B\in \mathsf M^{\langle 2\rangle}\cap \mathsf M^{reg}$, then $B$ and $-B$ are both regular and have the same Jordan blocks and the same eigenvalues. Hence $B$ and $-B$ are conjugate.
\end{proof}

Let $\fe^J(\mathsf M)$ denote the Jordan commuting variety and $p:\fe^J(\mathsf M)\to \mathsf M$ 
the projection to the first factor. The previous analysis implies that
\[
\dim p^{-1}(\mathsf M^{\langle 2\rangle}\cap \mathsf M^{reg})=\dim \mathsf \mathsf M^{\langle 2\rangle}+n=n^2+[n/2] .
\]
Thus, $\dim \fe^J(\mathsf M)\ge n^2+[n/2] > \dim\mathsf M$.
One can prove that this yields an irreducible component of maximal dimension; that is, 
$\dim  \fe^J(\mathsf M)= n^2+[n/2]$.
\end{ex}

The next table contains information on the restricted root systems associated with Jordan triads.
For a Hermitian involution $\sigma$,
we point out Lie algebras $\g$, $\h=\g^\sigma$, $\g_{00}=\ka$, the restricted root systems
$\Psi(G/H)$ and $\Psi(H/G_{00})$, and the multiplicity of the short roots in $\Psi(G/H)$, 
denoted $m_{\text{short}}$. For all items in the table, the multiplicity of long roots in $\Psi(G/H)$
equals $1$ and $\Psi(H/G_{00})$ is embedded in $\Psi(G/H)$ as a subset of {\sl short} roots.

\begin{center}
\begin{tabular}{c|>{$}c<{$}>{$}c<{$}>{$}c<{$}>{$}c<{$}c>{$}c<{$}}
   & \g    & \h        & \g_{00} & \Psi(G/H) & $m_{\text{short}}$ & \Psi(H/G_{00}) \\ \hline
1 & \sltn & \sln\oplus\sln\oplus\te_1 & \sln & \GR{C}{n} & 2 & \GR{A}{n-1} \\
2 & \spn & \gln    &  \son   & \GR{C}{n} & 1 & \GR{A}{n-1} \\ 
3 & \mathfrak{so}_{4n} & \mathfrak{gl}_{2n} & \spn & \GR{C}{n} & 4 & \GR{A}{n-1} \\
4 & \mathfrak{so}_{n+2} & \son\oplus\mathfrak{so}_2 &  \mathfrak{so}_{n-1} & \GR{C}{2} & $n-2$ & \GR{A}{1}\\
5 & \GR{E}{7} & \GR{E}{6}\oplus\te_1 & \GR{F}{4} &  \GR{C}{3} & 8 & \GR{A}{2} \\
\hline
\end{tabular}
\end{center}
\vskip2ex

\noindent
The root system of type $\GR{C}{n}$ has some short roots
that are not roots of $\GR{A}{n-1}$. Therefore, 
%if $m_{\textrm{short}}>1$, then
Proposition~\ref{prop:larger-comp} guarantees the existence of a subvariety in $\fe(\eus J)$
of dimension $\dim \eus J+m_{\text{short}}-1$, which is larger than the dimension of a generic 
fibre if $m_{\textrm{short}}>1$. However, a clever choice of $\tilde\ce\subset \ce_{11}$ 
(cf. Remark~\ref{rmk:obv-modific}(2)) allows to get a better lower bound on $\dim\fe(\eus J)$:

\begin{prop}   \label{prop:larger-comp2}
For all items in the table, we have $\dim\fe(\eus J)\ge \dim\eus J+(m_{\text{\rm short}}-1)[r/2]$,  
where $r$ is the rank of\/ $\Psi(G/H)$.
\end{prop}
\begin{proof}
Using Theorem~\ref{thm:jord-com-var}, we identify $\fe(\eus J)$ with the zero fibre of
the quadratic covariant $\g_{10}\times\g_{11}\to\g_{10}$ and work in the setting of Section~\ref{sect4}.
Let $\esi_1,\dots,\esi_r$ be the usual basis of $\mathfrak X(C_{11})\otimes \BQ$ such that
the roots of $\Psi(G/H)$ are $\pm\esi_i\pm\esi_j$ ($i\ne j$) and $\pm2\esi_i$. The roots in
$\Psi(H/G_{00})$ are $\pm(\esi_i-\esi_j)$. Therefore, $\g_{10}\oplus\g_{01}$ is the sum of root
spaces corresponding to $\pm(\esi_i+\esi_j)$ and $\pm2\esi_i$.
Set 
\[
    \tilde\ce=\{x\in \ce_{11}\mid (\esi_i+\esi_{r+1-i})(x)=0, \ \ i=1,2,\dots,
    \textstyle \left[\frac{r+1}{2}\right]\} .
\]
Then  $\dim\tilde\ce=[r/2]$, and we have $2[r/2]$ short roots of  $\g_{10}\oplus\g_{01}$
vanishing on $\tilde\ce$. Moreover, if $r$ is odd, then the long roots $\pm 2\esi_{[r+1/2]}$
also vanish on $\tilde\ce$. Therefore,
\[
\dim\z_\g(\tilde\ce)_{10}=\frac{1}{2}\dim\bigl(\z_\g(\tilde\ce)\cap(\g_{10}\oplus\g_{01})\bigr)=
\begin{cases}
 \text{$m_{\text{short}}{\cdot}\frac{r}{2}$} & \text{if $r$ is even} , \\
\text{$m_{\text{short}}{\cdot}\left[\frac{r}{2}\right]+1$} &  \text{if $r$ is odd} .
\end{cases}
\]
In both cases, this yields $\dim G_{00}{\cdot}(\z_\g(\tilde\ce)_{10}\oplus\tilde\ce)=
\dim\g_{11}+(m_{\text{\rm short}}-1)[r/2]$.
\end{proof}

For the Jordan algebra of all matrices (related to a Hermitian involution of $\sltn$), the above 
construction of $\tilde\ce$ gives exactly the subvariety of 
Example~\ref{ex:jordan-algebra-matrizen}. It is plausible that the lower bound of Proposition~\ref{prop:larger-comp2} provides the exact value of $\dim(\fe(\eus J))$.

\begin{rmk}
It is curious to observe that, for all Hermitian involutions leading to Jordan triads, the restricted root system is of type $\GR{C}{n}$; whereas, for all other Hermitian involutions, the restricted root
system $\Psi$ is of type $\GR{BC}{n}$. Namely, the symmetric pairs 
$\mathfrak{gl}_{n+m}\supset \gln\times\mathfrak{gl}_{m}\times\te_1$ ($n < m$) and
$\mathfrak{so}_{4n+2}\supset \mathfrak{gl}_{2n+1}$ lead to $\Psi\simeq \GR{BC}{n}$; and 
$\GR{E}{6}\supset\GR{D}{5}\times\te_1$ leads to $\Psi\simeq \GR{BC}{2}$.
\end{rmk}

\appendix
\section{Computations in classical Lie algebras} 
\label{app:A}

\noindent
Here we provide some computations related to the proof of Theorem~\ref{thm:strange-ineq}
for nilpotent elements in classical Lie algebras.

Let $\boldsymbol{\lb}=(\lb_1,\dots,\lb_s)$ be a partition %of $\sum\lb_i$ 
and $e\in\gln$ a nilpotent element corresponding to $\boldsymbol{\lb}$, also denoted 
by $e\sim \boldsymbol{\lb}$. 
Then $\sum\lb_i=n$ and 
\beq        \label{gl-formula}
\dim(\gln)^e=n+2 \sum_{i< j} \min\{\lb_i,\lb_j\}, \quad  
\dim(\sln)^e=\dim(\gln)^e-1 .%=\sum_{i=1}^t \hat\lb_i^2 
\eeq
If $e$ is a nilpotent element in $\son$ or $\spn$, with respective parity conditions on
$\boldsymbol{\lb}$, then
\begin{gather}
\dim(\spn)^e=\frac{\dim(\mathfrak{gl}_{2n})^e+\#\{i \mid \lb_i \text{ is odd}\}}{2} , \label{sp-formula} \\
\dim(\son)^e=\frac{\dim(\gln)^e-\#\{i \mid \lb_i \text{ is odd}\}}{2} .  \label{so-formula}
\end{gather}
See \cite[(3.8)]{hess76} and  \cite[2.4]{KP82}.
Below, we consider several symmetric pairs with classical $\g$ and check that \eqref{eq:strange-ineq}
is satisfied for all nonzero nilpotent elements of $\g_0$. There is no need in considering only non-even nilpotent element in $\g_0$, since the computations go through without this assumption.

\subsection{$(\g,\g_0)=(\sln,\son)$}
If $e\in\son$ and $e\sim \boldsymbol{\lb}$, then using \eqref{gl-formula} and
\eqref{so-formula} yields
\[
 \dim\g_0^e=\frac{\dim(\gln)^e-\# \{i \mid \lb_i \text{ is odd}\}}{2} ,\ \ 
 \dim\g_1^e=\frac{\dim(\gln)^e+\# \{i \mid \lb_i \text{ is odd}\}}{2} -1 .
\]
Therefore,
$ \dim \g_0^e-\dim\g_1^e +(n-1)= n-\# \{i \mid \lb_i \text{ is odd}\}$.
Here the parity condition means that each even part of $\boldsymbol{\lb}$ occurs an even number of times. Since $e\ne 0$, i.e., $\boldsymbol{\lb}\ne (1,\dots,1)$, the minimal value is $2$, and
it is attained for  $\boldsymbol{\lb}=(3,1^{n-3})$.

\subsection{$(\g,\g_0)=(\spn,\gln)$}   \label{subs:spn-gln}
 If $e\in\gln$ and $e\sim \boldsymbol{\lb}$, then 
the partition of $e$ as element of $\spn$ is obtained by doubling $\boldsymbol{\lb}$,  i.e.,
each part $\lb_i$ is replaced with $(\lb_i,\lb_i)$.
Then 
$\dim\g_0^e=\dim(\gln)^e$ is given by \eqref{gl-formula}, and using \eqref{sp-formula}
yields 
$\dim\g_1^e=2\left[\frac{\lb_i+1}{2}\right] + 2 \sum_{i< j} \min\{\lb_i,\lb_j\}$.
Hence
\[
    \dim \g_0^e-\dim\g_1^e +n=2n - 2\left[\frac{\lb_i+1}{2}\right]= n- \#\{i \mid \lb_i \text{ is odd}\}.
\]
For $e\ne 0$, the minimal value $2$ is attained for $\boldsymbol{\lb}=(2,1^{n-2})$ or
$(3,1^{n-3})$.

\subsection{$(\g,\g_0)=(\sone,\gln)$}   \label{subs:son-gln}
If $e\in\gln$ and $e\sim \boldsymbol{\lb}$, then $\dim\g_0^e=\dim(\gln)^e$ is again
given by \eqref{gl-formula}, while using this time \eqref{so-formula}, we obtain
$\dim\g_1^e=2\left[\frac{\lb_i}{2}\right] + 2 \sum_{i< j} \min\{\lb_i,\lb_j\}$.
Hence the result is even better than in the previous case. Indeed, we have here
$\dim\g_0^e-\dim\g_1^e\ge 0$.

\subsection{$(\g,\g_0)=(\mathfrak{sl}_{n+m},\sln\times\mathfrak{sl}_m\times\te_1)$}   \label{subs:dva-sln}
Here $n,m\ge 1$.
A nilpotent element $e\in\g_0$ is determined by two partitions,
$e\sim (\boldsymbol{\lb}; \boldsymbol{\mu})=((\lb_1,\dots,\lb_k);(\mu_1,\dots,\mu_s))$.
Using \eqref{gl-formula}, we obtain 
\begin{gather*}
\dim\g_0^e=n+m-1+2 \sum_{i< j} \min\{\lb_i,\lb_j\}+2 \sum_{i< j} \min\{\mu_i,\mu_j\} ,
\\
\dim\g_1^e=2 \sum_{i, j} \min\{\lb_i,\mu_j\}  .
\end{gather*}
\begin{multline*}
\text{Therefore,}\quad \dim \g_0^e-\dim\g_1^e +(n+m-1)=
    \\  2\bigl(n+m-1+\sum_{i< j} \min\{\lb_i,\lb_j\}+ 
    \sum_{i< j} \min\{\mu_i,\mu_j\}-\sum_{i, j} \min\{\lb_i,\mu_j\}\bigr) .
\end{multline*}
Since $n=\sum_i\lb_i$, $m=\sum_j\mu_j$, and $\sum_{i< j} \min\{\lb_i,\lb_j\}=
\sum_{i\ge 2}(i-1)\lb_i$, half of the RHS equals
\[
     \cF(\boldsymbol{\lb}; \boldsymbol{\mu}):
     = \sum_{i=1}^k i\lb_i + \sum_{j=1}^s j\mu_j -1- \sum_{i=1}^k \sum_{j=1}^s \min\{\lb_i,\mu_j\}.
\]
Arguing by induction, we prove that $\cF(\boldsymbol{\lb}; \boldsymbol{\mu})\ge 0$ for all 
$\boldsymbol{\lb}$ and $\boldsymbol{\mu}$, and 
if $n+m\ge 3$, then $\cF(\boldsymbol{\lb}; \boldsymbol{\mu}) > 0$.

1$^o$.  First,
$\cF(1^n; 1^m)=(n-m)^2/2 +(n+m)/2 -1$, which is positive if $(n,m)\ne(1,1)$.

2$^o$. The inequality is easily verified, if $\boldsymbol{\lb}$ or $\boldsymbol{\mu}$ consists of only one part.  

3$^o$. Suppose that $k\ge 2$ and $s\ge 2$.
Write $\boldsymbol{\lb}=(\lb_1,\boldsymbol{\lb}')$ and $\boldsymbol{\mu}=(\mu_1,\boldsymbol{\mu}')$. Then
\begin{multline*}
   \cF(\boldsymbol{\lb}; \boldsymbol{\mu})=\cF(\boldsymbol{\lb}'; \boldsymbol{\mu}')+
   \max\{\lb_1,\mu_1\}+ \sum_{i\ge 2}(\lb_i-\min\{\lb_i,\mu_1\})+\sum_{j\ge 2}(\mu_j-\min\{\lb_1,\mu_j\})
   \\
   \ge \cF(\boldsymbol{\lb}'; \boldsymbol{\mu}')+\max\{\lb_1,\mu_1\} \ge \max\{\lb_1,\mu_1\}.
\end{multline*}
Here $ \max\{\lb_1,\mu_1\}$ arises as $\lb_1+\mu_1-\min\{\lb_1,\mu_1\}$.

We omit computations related to the remaining classical symmetric pairs $(\sltn,\spn)$,
$(\mathfrak{sp}_{2n+2m}, \spn\times \mathfrak{sp}_{2m})$, and
$(\mathfrak{so}_{n+m}, \son\times \mathfrak{so}_{m})$.

\vskip1ex
\noindent
{\small  {\bf Acknowledgements.} Part of this work was done while I was visiting the Max-Planck Institut f\"ur Mathematik (Bonn).}

\end{document}